\documentclass[11pt, a4paper, reqno]{amsart}
\usepackage{amsmath, amsthm, amssymb, dsfont, color, fancyhdr, graphicx, enumerate, subfigure}
 \usepackage{setspace}
 
\usepackage[margin=3.2cm]{geometry}
\usepackage[arrow, matrix, curve]{xy}

\theoremstyle{plain}
\newtheorem{thm}{Theorem}[section]

\newtheorem{lem}[thm]{Lemma}
\newtheorem{prop}[thm]{Proposition}
\newtheorem{theo}[thm]{Theorem}
\newtheorem{cor}[thm]{Corollary}

\theoremstyle{definition}
\newtheorem{note}[thm]{Notation}
\newtheorem{ex}[thm]{Example}
\newtheorem{defn}[thm]{Definition}
\newtheorem{rk}[thm]{Remark}

\title{Conjugation in Brauer algebras and applications to character theory}

\author{Armin Shalile}

\address{Mathematical Institute, University of Oxford, 24-29 St Giles',
Oxford OX1 3LB, United Kingdom.}

\date{November 27, 2009}

\newcommand{\C}{\mathbb{C}}

\newcommand{\Z}{\mathbb{Z}}

\newcommand{\N}{\mathbb{N}}

\newcommand{\inv}{^{-1}}

\newcommand{\glnc}{GL_n(\mathbb{C})}

\newcommand{\brd}{B_r (\delta)}
\newcommand{\arc}{\ar@{-}@/_/}  
\newcommand{\tra}{\ar@{-}}      
\newcommand{\dta}{\ar@{.}}      
\newcommand{\ltra}{\ar@{<-}}      
\newcommand{\rtra}{\ar@{->}}      
\newcommand{\larc}{\ar@{<-}@/_/}  
\newcommand{\rarc}{\ar@{<-}@/_/}

\newcommand{\tr}{\text{tr}}

\newcommand{\res}{\downarrow_{\Sigma_r}}

\newcommand{\sta}{^{\ast}}

\begin{document}

\begin{abstract}
The Brauer algebra has a basis of diagrams and these generate a monoid $H$ consisting of scalar multiples of diagrams. Following a recent paper by Kudryavtseva and Mazorchuk, we define and completely determine three types of conjugation in $H$. We are thus able to define modular characters for Brauer algebras which share many of the properties of Brauer characters defined for finite groups over a field of prime characteristic. Furthermore, we reformulate and extend the theory of ordinary characters for Brauer algebras as introduced by Ram to the case when the Brauer algebra is not semisimple.
\end{abstract}

\maketitle

\section*{Introduction}
In \cite{Schur}, Issai Schur proved a fundamental link between representations of symmetric groups and general linear groups known as Schur-Weyl duality.  This duality arises from the following double centralizer property: Let $r$  and $n$ be  natural numbers, denote by $\Sigma_r$ the symmetric group on $r$ letters, by $\glnc$ the general linear group of invertible $n \times n$ matrices over $\C$ and suppose $E$ is an $n$-dimensional complex vector space.  Then $\Sigma_r$ and $\glnc$ act on the $r$-fold tensor space $E^{\otimes r}$ by place permutations and diagonal extension of the natural action on $E$, respectively. The $\C$-endomorphisms induced by the action of $\Sigma_r$ and $\glnc$ on  $E^{\otimes r}$ are precisely those $\C$-endomorphism which commute with $\glnc$ and $\Sigma_r$, respectively. 

In \cite{Brauer}, Richard Brauer defined an algebra which replaces the symmetric group in Schur-Weyl duality if one replaces the general linear group by its orthogonal and symplectic subgroups. This algebra is known today as the Brauer algebra and will be defined in Section~\ref{section defn of Brauer algebras}. 

Let $F$ be a field and suppose $\delta \in F$ is a fixed parameter. The Brauer algebra has a basis of diagrams and the multiplication of two diagrams yields another diagram, potentially scaled by a power of the parameter $\delta$. Since multiplication is associative, this gives rise to a possibly infinite semigroup $H$ generated by diagrams and their products. In \cite{Mazorchuk}, three notions of conjugacy in semigroups are defined and we  completely describe these in the case of $H$. 

The first type is via conjugation by the group of units in $H$. The group of units of $H$ can be identified with permutations and potentially scalar multiples of permutations in case $\delta$ has finite order. This is the strongest of the three notions of conjugacy and will imply the other types of conjugacy. In order to determine the conjugacy classes, we introduce a new invariant, the generalized cycle type and show that two diagrams are conjugate by a permutation if and only if they have equivalent generalized cycle types, see Theorem~\ref{theorem sigma conjugate iff gct the same}.

The second type of conjugation, called $(uv,vu)$-conjugation, arises from taking the transitive closure of the following relation on $H$: Two elements $h_1, h_2 \in H$ are related if there are $u,v \in H$ such that $h_1=uv$ and $h_2=vu$. It turns out that the conjugacy classes in this case have a nice description in terms of the cycle type, an invariant introduced by Arun Ram in \cite{Ram}. However, proving this poses  far more challenges than the first kind of conjugation. We will first show that diagrams with the same cycle type are $(uv,vu)$-conjugate by using the control the generalized cycle type gives us over the combinatorics. 

As a consequence of the study of the combinatorics, we get some information about eigenvalues of matrix representations. This will eventually allow us to define modular characters for Brauer algebras which are a generalization of the Brauer character of a finite group over fields of prime characteristic, see Definition~\ref{defn brauer chars}. Furthermore, we will reformulate and extend the ordinary character theory studied by Ram in the semisimple case to the non-semisimple case. In Theorem~\ref{thm modular char of  quasipartitions}, we will show that the relationship between ordinary characters of Brauer algebras and symmetric groups, observed by Ram in \cite{Ram}, also carries over to the non-semisimple case and even to modular characters. This will imply that in order to compute ordinary and modular characters, and thus decomposition matrices, it is sufficient to understand the restriction of the simple Brauer algebra modules to the symmetric group.

In Theorem~\ref{theo determination of chi conjugacy} we use our knowledge of characters to completely describe the third kind of conjugation: Two elements $h_1$ and $h_2$ are $\chi$-conjugate if and only if for all Brauer algebra characters $\chi$ we have $\chi(h_1)=\chi(h_2)$. Since $\chi$-conjugacy is the weakest of the three types of conjugacy, this will also allow us to complete our study of $(uv,vu)$-conjugation, see Corollary~\ref{cor determination of uvvu conjugacy}.

We remark that, similar to the case of finite groups, the knowledge of conjugacy classes gives a lot of information about the center of the Brauer algebra. In particular, we are able to provide an explicit algorithm for the determination of a basis of the center for all but finitely many values of the parameter, see \cite{Shalile}.

\section{Preliminaries}

\subsection{Definition of Brauer algebras}
\label{section defn of Brauer algebras}
Let $r$ be a natural number. By a diagram on $2r$ dots we mean a $2 \times r$-array of dots, that is two
rows of $r$ dots each, and each dot is connected by an edge to
exactly one other dot distinct from itself. An example of a diagram
on $16$ dots is $$ \begin{array}{c}
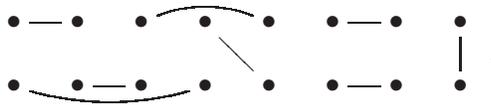
 \begin{xy}
\xymatrix@!=0.01pc{ \bullet \ar@{-}[r] & \bullet &\bullet \ar@{-}@/^/[rr] & \bullet \ar@{-}[rd]&\bullet &\bullet \ar@{-}[r] & \bullet & \bullet\ar@{-}[d] \\
\bullet \ar@{-}@/_/[rrr] & \bullet \ar@{-}[r] & \bullet & \bullet &
\bullet& \bullet\ar@{-}[r] & \bullet& \bullet }
\end{xy}\end{array}.$$ A horizontal arc is an edge connecting two points in the same
row, while a through arc is an edge connecting two dots in
different rows. A vertical arc is a special case of a
through arc connecting two dots in the same column. 

Given a field $F$ and a parameter $\delta \in F$, it is also
possible to compose diagrams with the same number of dots called concatenation. To concatenate two diagrams $a$ and $b$, write $a$ on top of $b$ and connect adjacent rows. The concatenation $a \circ b$ is the diagram obtained by deleting the loops of this construction, premultiplied by as many powers of $\delta$ as there were loops removed.

For example 

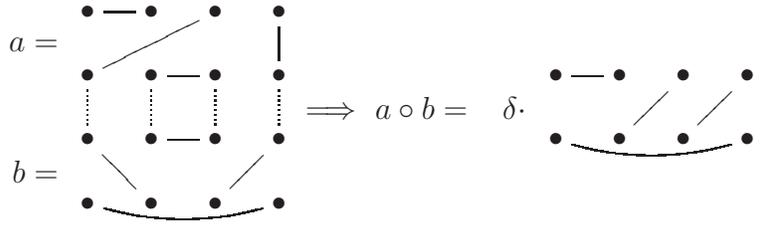
\begin{center}
 $\begin{array}{c}\begin{xy}
\xymatrix@!=0.01pc{ \bullet\tra[r] & \bullet &\bullet\tra[lld] &  \bullet \tra[d] \\
\bullet \dta[d]  \ar@{}[u]^{\text{ \large $a =$ \  }}  &
\bullet \tra[r] \dta[d] & \bullet \dta[d] & \bullet  \dta[d] \ar@{}[d]^{\text{ \large \ $\Longrightarrow \ a \circ b = $   }}  \\ \bullet \tra[rd] & \bullet \tra[r] &\bullet &  \bullet  \tra[ld] \\
\bullet \ar@{}[u]^{\text{ \large $b =$ \  }}  \arc[rrr] & \bullet &
\bullet & \bullet}
\end{xy}
 \begin{xy}
\xymatrix@!=0.01pc{ \ & \ &\ & \ \\
\bullet \tra[r]  &
\bullet & \bullet \tra[ld] & \bullet   \tra[ld]   \\ \bullet \arc[rrr]  \ar@{}[u]^{\text{ \large $\delta \cdot $ \  }} & \bullet  &\bullet &  \bullet  \\
\ & \ &\ & \ }
\end{xy} \end{array}.  $
\end{center}

\begin{defn}
Let $F$ be a field and fix some $\delta \in F$. For any $r \in \N$ define the Brauer algebra $\brd$ to be the associative $F$-algebra with basis all Brauer diagrams on $2r$ dots and multiplication given by concatenation.
\end{defn}

\begin{rk}
\begin{enumerate}[(i)]
\item If $F$ is a field of rational polynomials in a single indeterminate $x$ and $\delta=x$, then we get the so called generic Brauer algebra.
\item If $\Sigma_r$ denotes the symmetric group of degree $r$, then any diagram on $2r$ dots without horizontal arcs can be identified with an element of $\Sigma_r$. It is easily checked that concatenation of diagrams
without horizontal arcs is the same as composing permutations. Thus the group
algebra $F\Sigma_r$ is a subalgebra of the Brauer
algebra.
\end{enumerate}
\end{rk}

\begin{defn}
\label{defn cycle}
We say that a diagram is a cycle if the top left dot is connected to
some dot in the second column and for all other columns one dot in
each column is connected to a dot in the next column. Here the
$r+1$th column is identified with the first column. A diagram is said
to be in cycle form if it consists of cycles.

The following is a diagram in cycle form: \begin{center}
$\begin{array}{c} 
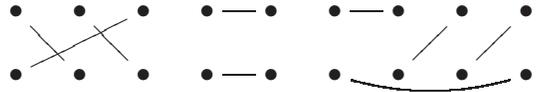
\begin{xy} \xymatrix@!=0.01pc{ \bullet \tra[rd] & \bullet \tra[rd] & \bullet & \bullet \tra[r] & \bullet & \bullet \tra[r] & \bullet  & \bullet & \bullet \\
\bullet \tra[rru]& \bullet & \bullet & \bullet\tra[r] & \bullet &
\bullet & \bullet \tra[ru] & \bullet \tra[ru] & \bullet 
\ar@{-}@/^/[lll]}
\end{xy} \end{array}.$
\end{center}

\end{defn}

\subsection{Semisimplicity}
 We will now briefly state when the Brauer algebra is semisimple. Semisimplicity depends on two factors: The characteristic of the field $F$ and the parameter $\delta$.

\begin{theo}[See Theorem 2.3 in \cite{Rui2}]
\label{thm when is br semisimple}
For $\delta \neq 0$ the Brauer algebra $B_r(\delta)$ is semisimple if and only if $$
\delta \notin \{ i \in \Z1 : 4-2r \leq i \leq r-2\} \backslash \{ i
\in \Z1 : 4-2r < i \leq 3-r, 2 \nmid i \} $$ and the characteristic of $F$ does not divide $n!$. Furthermore, $B_r(0)$  is semisimple if and only if $r=1,3,5$ and the characteristic of $F$ does not divide $n!$.
\end{theo}

\subsection{Cell modules}
\label{section defn of cell modules}
For any $n \in \N$, we call $\lambda=(\lambda_1,\lambda_2,\ldots,\lambda_k)$, with $k \in \N$ and $\lambda_i \in \N$ for $i=1,\ldots,k$,  a partition of $n$, if $\lambda_i > \lambda_{i+1}$ for each $i=1,\ldots, k-1$ and $|\lambda|:=\sum_{i=1}^k \lambda_i = n$. The natural numbers $\lambda_i$ are called the parts of $\lambda$. If we allow $\lambda_i \in \N_0$, then $\lambda$ is called a partition including $0$ parts. We consider $\lambda=()$ to be the unique partition of $0$. If $p \in \N$ is a prime then the partition $\lambda$ is called $p$-singular if there is $i \leq k-p+1$ such that $\lambda_i=\lambda_{i+1}=\ldots=\lambda_{i+p-1}$. Otherwise, it is called $p$-regular. If $\lambda_i$ is not divisible by $p$ for all $i=1,\ldots,k$ then $\lambda$ is called $p$-special. For convenience, we will adopt the convention that all partitions are $0$-regular and $0$-special. 

For every partition $\lambda$ of some natural number $n$, there is a distinguished $F\Sigma_n$-module $S^{\lambda}$, called the Specht module corresponding to $\lambda$. The set   $\{ S^{\lambda} \ | \ \lambda \text{ a  partition of $n$} \}$ is a complete set of representatives of simple $F\Sigma_n$-modules in characteristic $0$. If $F$ is a field of characteristic $p$ and $\lambda$ is $p$-regular then the module $S^{\lambda}$ has a unique simple top $D^{\lambda}$. Furthermore, the set $\{ D^{\lambda} \ | \ \lambda \text{ a $p$-regular partition of $n$} \}$ is a complete set of representatives of simple $F\Sigma_n$-modules in characteristic $p$. Similar modules exist for the Brauer algebra and we will briefly define
these modules. The following construction already implicitly appeared   in the work of Brown, see \cite{Brown}. However, we will follow the exposition of Section 2.2 in \cite{HenkePaget}. As shown in \cite{GrahamLehrer} the Brauer algebra is a cellular algebra and the $\brd$-modules which we will construct coincide with the cell modules. We need the following definition:
\begin{defn}
Let $t \in \N_0$ with $r-2t \geq 0$. A partial diagram on $r$ dots with $t$ arcs is a $1 \times r$ array
with $t$ horizontal arcs such that each dot is connected
to at most one other dot. The $r-2t$ dots which are not connected by an
edge are called free dots. The vector space with basis all partial
diagrams on $r$ dots with $t$ arcs is denoted by $V_{r,t}$, or $V_t$
if $r$ is clear from the context. An example of a partial diagram on $8$ dots with 2 arcs is:
$$\begin{array}{c}
 \begin{xy}
\xymatrix@!=0.01pc{ \bullet \ar@{-}@/^/[rrr] & \bullet  & \bullet & \bullet &
\bullet& \bullet\ar@{-}[r] & \bullet& \bullet }
\end{xy}\end{array}.$$
\end{defn}

The Brauer algebra acts on partial diagrams. If $a$ is a Brauer diagram on $2r$ dots and $v$ is a
partial diagram on $r$ dots then $v \circ a$ is determined as follows:
Write $v$ on top of $a$ and identify adjacent rows. If $n$ loops are created in this construction then the composition $v \circ a$ is equal to $\delta^n$ times the partial diagram obtained by
taking the bottom row of $a$ and following the resulting paths. If, however,
the number of arcs of the composition is increased then the product is $0$. For example, suppose $v$ and $a$ are given by
$$\begin{array}{c}
v =  \\
\ \\
\ \\
a =  \\
\
\end{array}
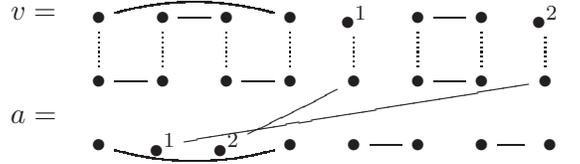
\begin{array}{c} \begin{xy}
\xymatrix@!=0.01pc{  \bullet \ar@{.}[d] \ar@{-}@/^/[rrr] & \bullet \ar@{-}[r]
\ar@{.}[d]
 & \bullet \ar@{.}[d] & \bullet \ar@{.}[d] &
\bullet^1 \ar@{.}[d] & \bullet\ar@{-}[r]  \ar@{.}[d] & \bullet \ar@{.}[d]  & \bullet^2 \ar@{.}[d] \\ \bullet \ar@{-}[r] & \bullet &\bullet \ar@{-}[r] & \bullet &\bullet \ar@{-}[lld] &\bullet \ar@{-}[r] & \bullet & \bullet\ar@{-}[lllllld] \\
\bullet  \ar@{-}@/_/[rrr] & \bullet^1  & \bullet^2 & \bullet  &
\bullet \ar@{-}[r] & \bullet & \bullet \ar@{-}[r] & \ \bullet  }
\end{xy} \end{array}.$$
Here we have already indicated the identification of the adjacent rows and also numbered certain dots which will be needed in a moment. Thus, 
 $$ \begin{array}{c}
   v \circ a = \delta^2
  \begin{xy}
\xymatrix@!=0.01pc{ \bullet \ar@{-}@/^/[rrr] & \bullet  & \bullet & \bullet &
\bullet \ar@{-}[r] & \bullet & \bullet \ar@{-}[r] & \bullet }
\end{xy} \end{array}.$$

Notice that the multiplication $ v \circ a$ induces a permutation $\pi(a,v)$ of the free dots of $v$ by joining them to free dots of the product $v \circ a$. In the example above
the first free dot of $v$ is joined to the second free dot of $v \circ a$
and the second free dot of $v$ is joined to the first free dot of
$v \circ a$ so we get the permutation $(1,2).$

We can thus define the cell modules $S^{(t,\lambda)}$:

\begin{defn}
\label{defnofaction}
For each $t \in \N_0$ such that $r-2t \geq 0$ and each partition $\lambda$ of $r-2t$, define the
$B_r(\delta)$-module $S^{(t,\lambda)}$ to be the vector space $ S^{(t,\lambda)} = S^{\lambda} \otimes_F V_{r,t}.$ The Brauer algebra acts on $S^{(t,\lambda)}$
by linear extension of the following action: For any $m \in S^{\lambda}$, any partial diagram $v \in V_t$ and any Brauer diagram $a \in B_r(\delta)$ set $$ (m \otimes
v) \cdot a = (\pi(a,v) m )\otimes (v \circ a).$$
\end{defn}

\begin{theo}
If $\brd$ is semisimple then the set  $$\{ S^{(t,\lambda)} \ | \  \text{$t \in \N_0, t \leq [r/2], \lambda$ a partition of $r-2t$} \}$$ is a complete set of representatives of simple $\brd$-modules. If $\brd$ is not semisimple and $\lambda$ is $p$-regular then $S^{(t,\lambda)}$ has a unique simple top $D^{(t,\lambda)}$ and the set  $$\{ D^{(t,\lambda)} \ | \  \text{$t \in \N_0, t \leq [r/2], \lambda$ a $p$-regular partition of $r-2t$} \}$$ is a complete set of representatives of simple $\brd$-modules.
\end{theo}

\begin{rk}
\label{rk indexing sets of simple modules}
Let $C$ denote the set of all partitions of $r-2k$ for $0 \leq k \leq [r/2]$. Furthermore, let $C_{p}$ denote the subset of $C$ consisting of all $p$-regular partitions where $p$ is the characteristic of the field $F$. The previous proposition implies that $C_{p}$ is the indexing set for simple $\brd$-modules and $C$ is the indexing set for the $\brd$-cell modules. 

If $C_{p'}$ is the set of all $p$-special partitions then there is a $1$-$1$-correspondence between $C_p$ and $C_{p'}$, and thus also between the simple $\brd$-modules and $p$-special partitions, see Lemma 10.2 in \cite{James}.
\end{rk}

\subsection{Decomposition matrices} 
\label{section defining decomposition matrices}
The $\brd$-cell modules  $S^{(t,\lambda)}$ can be defined over any field but  they will not necessarily be simple. The decomposition matrix of $\brd$ records how the cell modules decompose into simple $\brd$-modules and is defined as follows: The rows are labelled by $C$, the columns are labelled by $C_{p}$ and if $\lambda \in C$ and $\mu \in C_{p}$ then the $(\lambda,\mu)$th entry of the decomposition matrix is the multiplicity of $D^{(\frac{r-|\mu|}{2},\mu)}$ in $S^{(\frac{r-|\lambda|}{2},\lambda)}$.

In \cite{HartmannPaget}, Hartmann and Paget showed that large parts of the decomposition matrices of Brauer algebras can be recovered from decomposition matrices of smaller Brauer algebras and symmetric groups. In particular, there is an ordering of the simple modules such that certain symmetric groups occur along the diagonal and the decomposition matrix of a smaller Brauer algebra at the bottom right. In the following proposition, denote for any algebra $A$ by $D_A$ the decomposition matrix of $A$, provided it exists.
\begin{theo}[See Propositions 3.3  and 5.2 in \cite{HartmannPaget}]
\label{theo form of dec matrix} 
There is an ordering of the cell modules and the simple modules of the Brauer algebra such that $$  D_{ \brd }=\left( \begin{array}{cc}
D_{F\Sigma_r} & 0  \\ 
 \ast & D_{B_{r-2}(\delta)}
\end{array}\right).$$
In particular,  there is an ordering of simple and cell modules such that $D_{\brd}$ is lower triangular with $1$'s on the diagonal.
\end{theo}

\subsection{Monoids related to Brauer algebras}
Our main object of study will be a monoid $H$ consisting of scalar multiples of diagrams:

\begin{defn}
Fix a field $F$, a natural number $r$ and a parameter $\delta \in F$. Let $D$ be the set of all diagrams on $2r$ dots, and define a monoid $H$ by setting $$ H=\{ \delta^k d \ | \ k \in \N_0, d \in D \},$$ with multiplication given by concatenation.
\end{defn}

\begin{rk}
\label{rk monoid indep of field}
\begin{enumerate}[(i)]
\item Unless the parameter $\delta$ is zero or a root of unity, $H$ will be an infinite monoid.
\item If $\delta \neq 0$ then up to isomorphism the monoid $H$ only depends on the multiplicative order of $\delta \in F^{\times}$ but not on the field itself. Similarly, all monoids $H$ corresponding to diagrams with the same number of dots and $\delta=0$ are isomorphic.
\end{enumerate}
\end{rk}

Inspired by Section 2 of \cite{Mazorchuk}, we will now define three types of conjugation in the monoid $H$:

\begin{defn}[See Section 2 of \cite{Mazorchuk}]
\begin{enumerate}[(i)]
\item Two elements $d_1,d_2 \in H$ are called $\Sigma_r$-conjugate if and only if there exists a permutation diagram $\sigma \in \Sigma_r$ such that $\sigma \inv d_1 \sigma = d_2$.
\item Define a symmetric and reflexive relation $\sim'$ on $H$ by setting $d_1 \sim' d_2$ for $d_1,d_2 \in H$ if and only if there are $u,v \in H$ such that $d_1=uv$ and $d_2=vu$. Let $\sim_{(uv,vu)}$ be the equivalence relation obtained by taking the transitive closure of $\sim'$. If $d_1 \sim_{(uv,vu)} d_2$, we will say that $d_1$ and $d_2$ are $(uv,vu)$-conjugate.
\item Two elements $d_1,d_2 \in H$ are called $\chi$-conjugate if and only if for every matrix representation $\rho: \brd \to M_n(F)$ of $\brd$ the traces of the matrices $\rho(d_1)$ and $\rho(d_2)$ coincide. Here $M_n(F)$ denotes the set of $n \times n$ matrices with entries in $F$.
\end{enumerate}
\end{defn}

\begin{rk}
\begin{enumerate}[(i)]
\item Observe that  $\Sigma_r$-conjugacy implies  $(uv,vu)$-conjugacy which in turn implies  $\chi$-conjugacy.
\item In \cite{Mazorchuk}, the first type of conjugation is defined to be by conjugation by the group of units of $H$. However, it is sufficient to consider conjugation by elements of $\Sigma_r$ viewed as a subset of  $\brd$: Any diagram with at least one horizontal arc cannot be invertible since the identity has no horizontal arcs but horizontal arcs are always preserved in the multiplication. It is clear that diagrams with no horizontal arcs are precisely permutations. When $\delta$ is non-zero, it might also be the case that a scalar multiple of a permutation is invertible. Since, however, scalars commute with diagrams, multiplying diagrams by scalars  has no effect on conjugation.
\item The definition of conjugation in terms of characters is different from the definition in Section $2$ of \cite{Mazorchuk}. This follows since the Brauer algebra is not necessarily equal to the semigroup algebra $FH$ unless $\delta=1$.
\item By Remark~\ref{rk monoid indep of field}, it is clear that the first two types of conjugation do not depend on the underlying field of $H$ but only on the order of $\delta$.
\end{enumerate}
\end{rk}

\subsection{Modular reduction and lifting of roots of unity}
\label{section on reduction}
The triple $(S,R,F)$ is called a $p$-modular system if $R$ is a complete discrete valuation ring with unique maximal ideal $\Pi$, $S$ is its field of fractions of characteristic $0$ and $F$ is the residue field $R/\Pi$ of characteristic $p$. For a detailed account of $p$-modular systems, we refer the reader to Chapters 11ff in \cite{Landrock}. 
We will assume throughout this paper that $S$ contains the $|\Sigma_r|$th roots of unity
and that $F=R/\Pi$ is algebraically closed. Denote by $\overline{ \phantom{u} }:R \to R/\Pi=F$ the natural projection modulo $\Pi$, that is $\overline{u}=u+\Pi$ for every $u \in R$.

We will need the fact that there is a unique way of lifting roots of unity from $F$ to $R$:

\begin{prop}[See for example Remark 40.2 in \cite{Kuelshammer}]
\label{prop roots of 1 lift uniquely}
Let $k$ be a natural number which is not divisible by $p$ but divides $|\Sigma_r|$. Furthermore, let $U$ be the set of $k$th roots of unity in $S$ and let $V$ be the set of $k$th roots of unity in $F$. Then $U \subseteq R$ and the restriction of the natural projection modulo $\Pi$ to $U$ gives an isomorphism between $U$ and $V$.
\end{prop}

Denote by $\hat{ \ }: V \to U$ the inverse map of the natural projection modulo $\Pi$, and fix a number $\hat{\delta} \in R$ such that $\overline{\hat{\delta}} = \delta \in F$. Notice that we use the $\hat{ \ }$-notation to signify that $\hat{\delta}$ is a lift of $\delta$ but it is not an image under the map $\hat{ \ }$.

\section{$\Sigma_r$-conjugacy classes}
In this section we will study the most classical form of conjugation in semigroups, namely by conjugation by the group of units of the semigroup. We will first only study conjugation in the diagram basis $D$ and extend this to the whole of $H$ later. In order to examine these conjugacy classes more closely we will introduce a new invariant, the generalized cycle type:

\begin{defn}
\label{defnGCT} For any diagram $d \in \brd$ define the generalized
cycle type (GCT) in the following way: Connect each dot in the top
row of $d$ with the dot in the bottom row below it. The connected
components $ \{ c_i \}$ of the resulting graph will be closed paths 
except in the special case of a
vertical arc. 

Associate to each connected component $c_i$ a string
of letters $s_i $ obtained in the following way: 
\begin{itemize}
\item If the connected component is a vertical arc then $s_i = T$. If not pick the top left dot in $c_i$ and follow the path
clock-wise.
\item Following the path, record in order for any through arc of the original diagram which is traversed a "$T$",
and for any horizontal arc of the original diagram a "$U$" if it was
in the upper row and an "$L$" if it was in the lower row. Continue this until you
get back to your starting point, see example below. The string
$s_i$ is then given by the recorded string of letters.
\end{itemize}
The generalized cycle type of the diagram $d$, denoted $\tau (d)$, is
then given by $$\tau(d) = \{ s_1,s_2, \ldots, s_q\}$$ where $q$ is
the total number of connected components.
\end{defn}

\begin{rk} This strengthens the notion
of a cycle type of a diagram as introduced by Ram in Section 2 of \cite{Ram} by recording more information about the types of arcs in each component. We will discuss the difference between these two notions in more detail in Section~\ref{section comparison ct and gct}.
\end{rk}

\begin{ex}
Consider the following diagram $d$: $$d=\begin{array}{c}
 \begin{xy}
\xymatrix@!=0.01pc{ \bullet \ar@{-}[r] & \bullet &\bullet \ar@{-}@/^/[rr] &
\bullet \ar@{-}[rd]&\bullet &\bullet \ar@{-}[r] & \bullet &
\bullet\ar@{-}[d] \\  \bullet \ar@{-}@/_/[rrr] & \bullet \ar@{-}[r]
& \bullet & \bullet & \bullet& \bullet\ar@{-}[r] & \bullet&  \bullet 
}
\end{xy} \end{array}.$$

Connecting top and bottom row we obtain: 
$$ \begin{array}{c} 
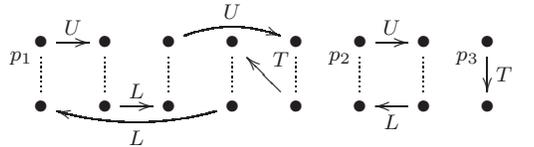
\begin{xy}
\xymatrix@!=0.01pc{  \bullet  \ar@{.}[d] \ar@{->}[r]^U  & \bullet \ar@{.}[d] &\bullet \ar@{.}[d] \ar@{->}@/^/[rr]^U & \bullet \ar@{.}[d] \ar@{<-}[rd]^T &\bullet \ar@{.}[d] &\bullet \ar@{.}[d] \ar@{->}[r]^U & \bullet \ar@{.}[d] & \bullet  \ar@{.}[d] \ar@{->}[d]^T \\
\bullet \ar@{}[u]^>{p_1} \ar@{<-}@/_/[rrr]_L & \bullet \ar@{->}[r]^L & \bullet & \bullet &
\bullet& \bullet\ar@{<-}[r]_L \ar@{}[u]^>{p_2} & \bullet &  \bullet \ar@{}[u]^>{p_3}}
\end{xy} \end{array}.$$ There are three connected components $c_1,c_2,c_3$, labelled from left to right. To compute $s_1$, we
start with the top left dot labelled $p_1$ and follow the path clockwise as indicated in the diagram. The first
edge we traverse is a horizontal arc in the upper row. This is followed by
two more horizontal arcs, one in the lower and one in the upper
row. Next a through arc is traversed and another horizontal arc in the
lower row. Thus, we conclude that $s_1 = ULUTL$. Similarly, we find $s_2 = UL$ and $s_3 = T$. Therefore the generalized cycle type of $d$ is given by $\{ULUTL,UL,T\}.$

\end{ex}

\begin{rk}
\label{rk Us and Ls alternate}
Notice that not all strings of letters from the alphabet $\{ U,L,T\}$ can occur. Suppose that while computing the string of some connected component an upper horizontal arc is traversed. Then the next edge must necessarily originate from the bottom row. Therefore, the edge can only be either an upward pointing through arc or a lower horizontal arc. If the edge which is traversed after an upper horizontal arc is an upward pointing through arc then the next edge will again be originating from the bottom row. It is thus clear that upper horizontal and lower horizontal arcs must alternate, so that for example a string $UTUL$ is impossible.
\end{rk}

For our purposes both the starting point as well as the direction
in which the paths are traversed are arbitrary:

\begin{defn}
Let $s = s_1 s_2 \ldots s_k$ be a string of $k$ letters with $s_j
\in \{ U,L,T \}$ for $j=1,\ldots,k$. Define two operations $\rho$ ("reversing") and
$\sigma$ ("shifting") by
\begin{align*} \rho(s)& = s_k s_{k-1} \ldots
s_1, \\ \sigma(s) & = s_2 s_3 \ldots s_k s_1.
\end{align*} These operations can be composed in the natural way and we have that  $ \rho \sigma = \sigma \inv \rho$. Two strings $s$ and $t$ are said to be equivalent if $t
= \sigma^n \rho^m (s)$ for $m,n \in \Z$, that is, if one string can be transformed into
the other by repeated shifting and reversing. Two generalized cycle types are said to
be equivalent if they contain the same strings up to equivalence.
\end{defn}

The following theorem justifies the definition of the generalized cycle type:

\begin{theo}
\label{theorem sigma conjugate iff gct the same}
Two diagrams in $\brd$ are $\Sigma_r$-conjugate if and only if their
generalized cycle types are equivalent.
\end{theo}

\begin{proof}
\ \ \\
($\Longrightarrow$) Suppose $d, d' \in \brd$ are $\Sigma_r$-conjugate, that is there is $\pi
\in \Sigma_r$ such that $d' = \pi \inv d \pi$. Connect each dot in
the top row of both $d$ and $d'$ with the dot in the bottom row
vertically below it as in Definition~\ref{defnGCT}. For simplicity
we will call these slightly modified graphs $d$ and $d'$ as well. We
will show that conjugation by $\pi$ induces a graph isomorphism
between $d$ and $d'$. We have to specify a map $\phi$ from the
vertex set of $d$ to the vertex set of $d'$ such that whenever there
is an edge between the vertices $u$ and $v$ of $d$, then there is
also an edge between $\phi(u)$ and $\phi(v)$.

Label the top row of $d$ with $1,\ldots,r$ in increasing order from
left to right and the bottom row with $1',\ldots,r'$ in increasing
order from left to right. Define the map $\phi$ in the following
way: For a dot $k$ of the top row of $d$ define $\phi(k)$ to be the
$\pi(k)$th dot of the top row of $d'$ and for the bottom row of $d$
define $\phi (k')$ to be the $\pi(k)$th dot in the bottom row of
$d'$, see Figure~\ref{figure part 1}.

 Figure~\ref{figure part 1} illustrates the definition of the map $\phi$. It shows
an extract from the multiplication $\pi \inv d \pi$. Following the
paths from row $1$ to row $6$ of this picture, yields precisely
$d'$. Notice that the permutations $\pi \inv$ is just the
permutation $\pi$ but horizontally flipped over. Hence $\phi(k)$ and $\phi(k')$
are in the same column and thus conjugation preserves vertical arcs.
In fact, conjugation by a permutation simply rearranges columns.

Notice that $\phi$ is clearly a bijection. It remains to show that
$\phi$ preserves edges. By construction each dot of $d$ is attached
to two edges: the original edge coming from the diagram and the
vertical edge added for the calculation of the GCT. We have already
shown above that the latter is clearly preserved. To show that the
original edge, that is the edge from the original diagram $d$ before
adding vertical lines, is preserved as well, there are two case to be
considered: Either the edge is a horizontal or a through arc. These cases are illustrated in Figure~\ref{figure part 2}.
\begin{figure}[t]
\label{figure graph isom}
 \subfigure[Conjugation permutes columns.]{ \label{figure part 1}$  \begin{xy}
\xymatrix@!=0.01pc{  1  & 2 & \ldots & k & \pi(k)  \ar@{}[d]^>{\phi(k)} & r   \\
            \bullet \ar@{}[d]_{ \text{   \large  $\pi  ^{-1}$   } } & \bullet &\ldots & \bullet &\bullet \tra[ld]  &\bullet   \\
            \bullet \ar@{}[d]^>{1}  \dta[d]  & \bullet \ar@{}[d]^>{2} \dta[d] &\ldots & \bullet \ar@{}[d]^>{k} \dta[d] &\bullet  \dta[d] &\bullet \ar@{}[d]^>{r}  \dta[d]  \\
            \bullet \ar@{}[d]_{ \text{   \large  $d  \phantom{^{-1}}$   } } \ar@{}[d]^>{1'} & \bullet \ar@{}[d]^>{2'} &\ldots & \bullet\ar@{}[d]^>{k'} &\bullet  &\bullet\ar@{}[d]^>{r'}   \\
            \bullet   \dta[d]  & \bullet  \dta[d] &\ldots  & \bullet \dta[d]  &\bullet  \dta[d] &\bullet \dta[d]   \\
            \bullet \ar@{}[d]_{ \text{   \large  $\pi  \phantom{^{-1}}$   } } & \bullet &\ldots & \bullet \tra[rd] &\bullet  &\bullet   \\
            \bullet   & \bullet &\ldots & \bullet &\bullet\ar@{}[u]_<{\phi(k')}  &\bullet }
\end{xy}$} \hfill \subfigure[Conjugation preserves the type of an edge.]{  \label{figure part 2}$  \begin{xy}
\xymatrix@!=0.01pc{  1  & \ldots & \pi(l) \ar@{}[d]^>{\phi(l)} & l & \ldots & k & \pi(k) \ar@{}[d]^>{\phi(k)} & r   \\
            \bullet   \ar@{}[d]_{ \text{   \large  $\pi \inv $   } } & \ldots & \bullet \tra[rd] & \bullet &\ldots & \bullet &\bullet \tra[ld]  &\bullet  \\
            \bullet    \dta[d]^>{1} & \ldots  & \bullet \dta[d] & \bullet \dta[d]  \ar@{}[d]^>{l} &\ldots & \bullet \dta[d] \ar@{}[d]^>{k} &\bullet  \dta[d] &\bullet \ar@{}[d]^>{r} \dta[d]  \\
            \bullet \ar@{}[d]_{ \text{   \large  $d \phantom{^{-1}}$   } } \ar@{}[d]^>{1'}   & \ldots &   \bullet & \bullet \ar@{}[d]^>{l'} \arc[rr] &\ldots & \bullet \ar@{}[d]^>{k'} &\bullet  &\bullet\ar@{}[d]^>{r'}    \\
            \bullet \dta[d] & \ldots &   \bullet \dta[d] & \bullet  \tra[rru] \dta[d] &\ldots  & \bullet \dta[d]  &\bullet  \dta[d] &\bullet \dta[d]   \\
            \bullet \ar@{}[d]_{ \text{   \large  $\pi  \phantom{^{-1}}$   } }  & \ldots &   \bullet & \bullet &\ldots & \bullet  &\bullet  & \bullet   \\
            \bullet   & \ldots &   \bullet \ar@{}[u]^<{\phi(l')} \tra[ru] & \bullet &\ldots & \bullet &\bullet  &\bullet  }
\end{xy}$}
 \caption{Schematic overview of the multiplication $d'=\pi \inv d \pi$.}
\end{figure}
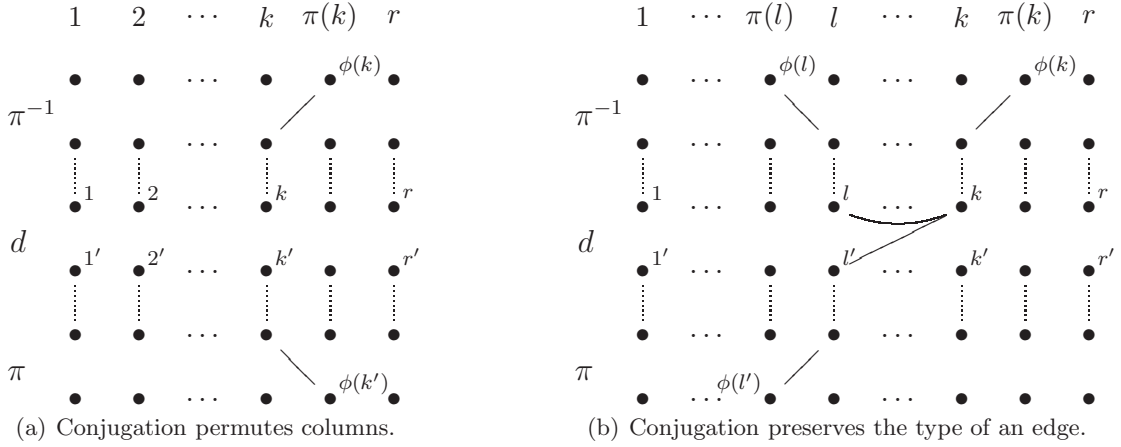

If the edge is horizontal say connects dots $l$ and dots $k$ in the
top row of $d$ then Figure~\ref{figure part 2} clearly shows that $\phi(l)$
and $\phi(k)$ are connected by an edge as well. Notice that $\phi$
hence preserves top and bottom arcs in the sense that if an arc is in the top/bottom row then its image will also be in the top/bottom row. The case of a
through arc where $k$ is connected to an edge $l'$ say in the bottom row is similar.
Again Figure~\ref{figure part 2} clearly shows that $\phi(k)$ and $\phi(l')$ are
connected by an edge.

So the two graphs are isomorphic and thus, in particular their
connected components. Notice that we have the additional property
that $\phi$ preserves upper horizontal, lower horizontal, vertical
and through arcs. Thus if $p_1,p_2,\ldots,p_l$ is a closed path in
$d$ where $p_i$ are distinct dots of $d$ and $p_1 = p_l$ then
$\phi(p_1),\phi(p_2),\ldots,\phi(p_l)$ is a closed path in $d'$ and
whenever the edge between $p_j$ and $p_{j+1}$ is a upper
horizontal/lower horizontal/vertical/through arc, then the edge
between $\phi(p_j)$ and $\phi(p_{j+1})$ will also be a upper
horizontal/lower horizontal/vertical/through arc. Since up to
equivalence the starting point and orientation of the path are
irrelevant, the GCT of both diagrams is the same, as required.

($\Longleftarrow$) Conversely, suppose two diagrams $d$ and $d'$ have the same GCT. We
have to show that they are conjugate. Notice that any diagram can be
put into cycle form via conjugation by a permutation: After adding
vertical lines as before, pick a connected component say with $2n$
dots and within this connected component pick any dot in the upper
row, say it is in column $c_1$. Following the path clockwise note
which columns are traversed in which order. If the order is $c_1,
c_2, c_3, \ldots, c_n$, conjugate by the permutation $\pi = \left(
                                           \begin{array}{cccc}
                                             c_1 & c_2 & \cdots & c_n \\
                                             1 & 2 & \cdots & n \\
                                           \end{array}
                                         \right)
$. Repeat this with the other connected components, of course
changing the permutation to $\pi' = \left(
                                           \begin{array}{cccc}
                                             c_1' & c_2' & \cdots & c_m' \\
                                             n+1 & n+2 & \cdots & n+m \\
                                           \end{array}
                                         \right)
$ this time, and so on. The diagram is now in cycle form. As shown
above, conjugation just rearranges columns. Thus the cycles of
$d$ and $d'$ can be rearranged by permutations to match each other.
Furthermore, we can also change the starting point and orientation of
each cycle arbitrarily by rearranging columns.
\end{proof}

Since conjugation by permutations does not create any loops and hence no additional powers of $\delta$, we can easily extend this to $H$:

\begin{cor}
Let $d_1,d_2 \in D$ and suppose $k,l \in \N_0$. Then $\delta^k d_1$ and $\delta^l d_2$ are $\Sigma_r$-conjugate if and only if $k=l$ and $d_1$ and $d_2$ have equivalent generalized cycle types.
\end{cor}

\section{Conjugacy classes under the $(uv,vu)$-relation}
In this section we will determine sufficient conditions for elements of $H$ to be $(uv,vu)$-conjugate.  The complete determination of the $(uv,vu)$-conjugacy classes will be postponed until Section~\ref{section chi conjugacy}. We will start by defining the invariant which will be essential.

\subsection{The cycle type of a diagram}
\label{section comparison ct and gct}

The cycle type is an invariant of diagrams which was introduced by Arun Ram in \cite{Ram} for the study of ordinary characters in the semisimple case. It is defined as follows:
\begin{defn}[Section 3 of \cite{Ram}]
For any diagram $d \in \brd$ define the cycle type (CT) in the following way: Connect each dot in the top
row of $d$ with the dot in the bottom row below it. The connected
components $ \{ c_i \}$ of the resulting graph will be closed paths except in the special case of a
vertical arc. Associate to each connected component $c_i$ a natural number
 $n_i $ obtained in the following way:
\begin{itemize} \item Give each connected component an orientation by starting with an arbitrary dot and following the path in an arbitrary direction. While traversing one of the original edges, that is an edge of the original diagram, record the direction in which the edge was traversed.
\item Count the number of upward and the number of downward pointing arcs.
\item The number $n_i$ is then given by the absolute value of the difference of upward and downward pointing arcs.
\end{itemize}
The  cycle type of the diagram $d$, denoted $\xi(d)$  is 
then given by the partition with parts $n_i$, counting zeros.
We will say that $\mu$ is a cycle type of $\brd$ if it is the cycle type of some diagram $d \in \brd$.
\end{defn}

\begin{ex}
Consider the following diagram $a$: $$
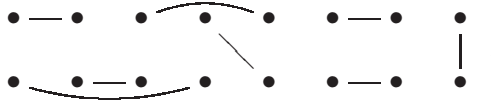
\begin{array}{c} \begin{xy} 
\xymatrix@!=0.01pc{ \bullet \ar@{-}[r] & \bullet &\bullet \ar@{-}@/^/[rr] &
\bullet \ar@{-}[rd]&\bullet &\bullet \ar@{-}[r] & \bullet &
\bullet\ar@{-}[d] \\  \bullet \ar@{-}@/_/[rrr] & \bullet \ar@{-}[r]
& \bullet & \bullet & \bullet& \bullet\ar@{-}[r] & \bullet& \bullet 
}
\end{xy} \end{array}.$$

Connecting top and bottom row and adding orientations to each connected component, we obtain:

$$\begin{array}{c} 
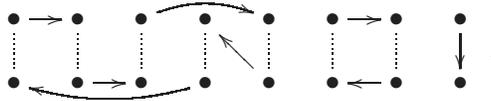
\begin{xy} 
\xymatrix@!=0.01pc{ \bullet \ar@{.}[d] \ar@{->}[r] & \bullet \ar@{.}[d] &\bullet \ar@{.}[d] \ar@{->}@/^/[rr] & \bullet \ar@{.}[d] \ar@{<-}[rd]&\bullet \ar@{.}[d] &\bullet \ar@{.}[d] \ar@{->}[r] & \bullet \ar@{.}[d] & \bullet \ar@{.}[d] \ar@{->}[d] \\
\bullet \ar@{<-}@/_/[rrr] & \bullet \ar@{->}[r] & \bullet & \bullet &
\bullet& \bullet\ar@{<-}[r] & \bullet& \bullet }
\end{xy} \end{array}. $$

There are three connected components $c_1,c_2,c_3$. We can see that in the first component there is only one edge pointing upward and no edge pointing downward, so $n_1=1$. Similarly, we can see that $n_2=0$ and $n_3=1$. Thus $\xi(a)=(1^2,0)$.
\end{ex}

\begin{rk}
\label{rk which ct can occur}
Notice that the cycle types which occur for a given Brauer algebra $\brd$ are always of the form $(\alpha,0^{m_0})$ where $\alpha$ is a list of non-zero natural numbers such that $(\alpha)$ is a partition of $r-2k$ for $0 \leq k \leq [r/2]$. Furthermore, the non-negative integer $m_0$ satisfies $m_0 \leq  \frac{r-|(\alpha)|}{2}$. 
\end{rk}

The following proposition relates cycle type and generalized cycle type:

\begin{prop}
\label{prop same gct implies same ct}
Suppose two diagrams have the same generalized cycle type. Then they also have the same cycle type.
\end{prop}

\begin{proof}
By Theorem~\ref{theorem sigma conjugate iff gct the same}, the two
diagrams must be $\Sigma_r$-conjugate. Recall that conjugation by a
permutation simply permutes the columns of a diagram, see the proof
of Theorem~\ref{theorem sigma conjugate iff gct the same}. Hence by definition of the cycle type it follows that the two
diagrams must have the same cycle type.
\end{proof}

\begin{rk}
\label{rk calculate CT from GCT}
\begin{enumerate}[(i)]
\item The converse is not true. For example the diagrams $$\begin{array}{c} 
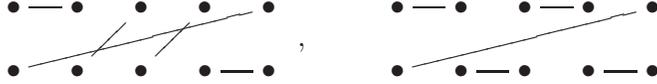
\begin{xy}  \xymatrix@!=0.01pc{ \bullet \tra[r] & \bullet  & \bullet \tra[ld] & \bullet \tra[ld]  & \bullet \tra[lllld]  \\ \bullet & \bullet  & \bullet  & \bullet\tra[r]  & \bullet  } \end{xy} \end{array}, \qquad \begin{array}{c} \begin{xy} \xymatrix@!=0.01pc{ \bullet \tra[r] & \bullet  & \bullet \tra[r] & \bullet  & \bullet \tra[lllld]  \\ \bullet & \bullet \tra[r] & \bullet  & \bullet  \tra[r]& \bullet  } \end{xy} \end{array}$$ have the same cycle type $(1)$, but do not even have the same number of horizontal arcs. Hence they cannot have the same generalized cycle type. There are also examples of diagrams with the same number of arcs in each connected component and equal cycle types, but with different generalized cycle types.

\item There is a simple method to calculate the cycle type from the generalized cycle type. By Remark~\ref{rk Us and Ls alternate} it is clear that if in a connected component a through arc is preceded by an upper horizontal arc, then the through arc must point upwards. Similarly, if a through arc is preceded by a lower horizontal arc then it must point downwards. Also, if a through arc is preceded by a through arc then both arcs point in the same direction.  For example, if the generalized cycle type is $\{ UTTLTUL \}$ then the first two through arcs are upward pointing and the third is downward pointing, so that the cycle type would be $(1)$.
\end{enumerate}
\end{rk}

\subsection{A necessary condition for $(uv,vu)$-conjugacy in $D$}
Again, we will first study $(uv,vu)$-conjugacy in $D$ and extend the result to $H$ at a later stage.  For technical reasons, we will work with a slightly stronger relation which will give us additional control of eigenvalues, as we shall see:

\begin{defn}
\label{defn of gct equiv rel}
Let $\sim_q$ be the reflexive and symmetric relation on $H$ defined by setting $d \sim_q d'$ for $d,d' \in H$ if and only if $d$ and $d'$ have the same cycle type and at least one of the following holds:
\begin{enumerate}[(i)]
\item The diagrams $d$ and $d'$ have the same generalized cycle type.
\item There is some $q \in H$ such that $d=qd$ and $d'=dq$.
\item There is some $q \in H$ such that $d'=qd'$ and $d=d'q$. 
\end{enumerate}
Define an equivalence relation $\sim$ by taking the transitive closure of $\sim_q$.

Since diagrams with the same generalized cycle type are $\sim$-equivalent, we can extend $\sim$ to an equivalence relation on generalized cycle types. Thus, we will say that two generalized cycle types $\tau_1$ and $\tau_2$ are equivalent, denoted $\tau_1 \sim \tau_2$, if and only if there are elements $d_1,d_2 \in D$ with generalized cycle types $\tau_1$ and $\tau_2$, respectively, such that $d_1 \sim d_2$.
\end{defn}

\begin{rk}
\begin{enumerate}[(i)]
\item Notice that if $d_1,d_2 \in H$ are $\Sigma_r$-conjugate then they are also $(uv,vu)$-conjugate. Thus elements with the same generalized cycle type are $(uv,vu)$-conjugate by Theorem~\ref{theorem sigma conjugate iff gct the same}. 
\item By the first remark and the definition of $\sim$, we have that if $d_1,d_2 \in H$ then $d_1 \sim d_2$ implies $d_1 \sim_{(uv,vu)} d_2$. 
\end{enumerate}
\end{rk}

Our first aim is to show that diagrams with the same cycle type are $\sim$-equivalent:

\begin{theo}
\label{theo ct determined uv conjugation}
Let $d_1,d_2 \in D$ be elements with the same cycle type. Then $d_1 \sim d_2$  and, in particular, $d_1$ and $d_2$ are $(uv,vu)$-conjugate.
\end{theo}

The proof of this theorem is quite combinatorial. It will involve performing certain operations on
generalized cycle types.  Assume that
$\mu=(\mu_1,\mu_2,\ldots,\mu_k,0^m)$ is the common cycle type (CT)
of the diagrams $d_1$ and $d_2$ where $\mu_i \neq 0$ for
$i=1,\ldots,k$. We will show that there is some diagram $d$ of
generalized cycle type (GCT) $$\{
(UL)^{r-|\mu|}T^{\mu_1},T^{\mu_2},T^{\mu_3}, \ldots,
T^{\mu_k},UL,UL,\ldots,UL \}$$ (where at the end $UL$ occurs $m$
times) which is $\sim$-equivalent to both of them.

We will prove the following relations which are going to be crucial:

\begin{lem}
\label{lem gct calculus} Let $Z$ and $Z'$ be strings of letters from
the alphabet $\{ U,L,T \}$, and let $\rho$ (reversing) and $\sigma$
(shifting) be as in Definition~\ref{defnGCT}. Then
\begin{enumerate}[(i)]
\item $ \{Z \} \sim \{ \rho(Z) \} \text{ \ and \ } \{ Z \} \sim \{ \sigma(Z) \}, $ \label{gctcalc1}
\item \label{gctcalc2}   $ \{ ULTZ\} \sim \{ULZT\},$
\item  \label{gctcalc3}   $\{TLTZ\} \sim \{LULZ\},$
\item   \label{gctcalc4}  $ \{ Z,ULZ'\} \sim \{ ZLU,Z'\}.$
\end{enumerate}
\end{lem}

\begin{proof}

\begin{enumerate}[(i)]
\item Since the GCT is only defined up to reversing and shifting, both $\{ \rho(Z) \} $ and $\{ \sigma(Z) \}$ have the same GCT as $\{ Z \}$. The first part thus follows.
\item Let $d$ have cycle type $\{ ULTZ \}$. Thus without loss of generality, $d$ is of the form
$$ \begin{array}{c} 
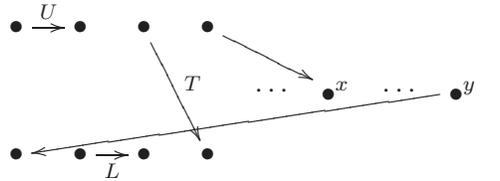
\begin{xy}   \xymatrix@!=0.01pc{    \bullet \rtra[r]^U & \bullet & \bullet \rtra[ddr]^T & \bullet \rtra[drr] & & \\  & & & & \cdots & \bullet^x & \cdots &\bullet^y \\ \bullet  \ltra[rrrrrrru] &\bullet \rtra[r]_L  & \bullet& \bullet & & } \end{xy} \end{array}.$$ Here we have only depicted the first four columns of the diagram and marked the two dots which are connected to the bottom left and top right dots of these first four columns by $y$ and $x$, respectively.
Furthermore, $x$ and $y$ have been drawn in the middle of the
diagram in order to show that no assumption has been made about
whether they are in the top row or the bottom row. We have also
drawn the orientation in which the GCT is to be read.

 The idea is to specify a diagram $q \in \brd$ such that $qd= d$ and such that $d':=dq$ has the same CT and the required GCT.
Let $q$ be the diagram $$\begin{array}{c} 
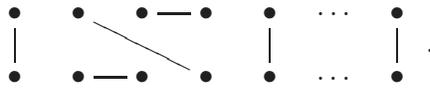
\begin{xy} \xymatrix@!=0.01pc{   \bullet \tra[d] &  \bullet \tra[drr] & \bullet  \ar@{-}[r] & \bullet &  \bullet \tra[d] & \ldots & \bullet \tra[d]   \\ \bullet  & \bullet \tra[r]  & \bullet    & \bullet    &  \bullet  & \ldots &  \bullet } \end{xy} \end{array}. $$
Then $qd$ is given by $$\begin{array}{c} 
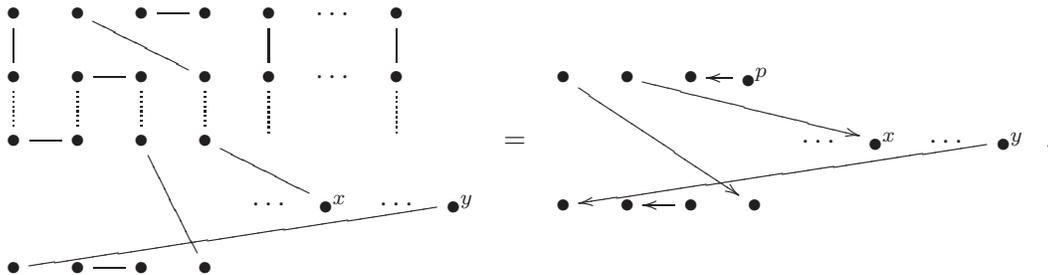
\begin{xy}   \xymatrix@!=0.01pc{     \bullet
\tra[d] &  \bullet \tra[drr] & \bullet  \ar@{-}[r] & \bullet &
\bullet \tra[d] & \ldots & \bullet \tra[d]   \\ \bullet \dta[d] &
\bullet \tra[r] \dta[d] & \bullet \dta[d]  & \bullet \dta[d] &
\bullet \dta[d] & \cdots &  \bullet \dta[d] & \\ \bullet \tra[r] &
\bullet & \bullet \tra[ddr] & \bullet \tra[drr] & & & \\  & & & &
\cdots & \bullet^x & \cdots &\bullet^y \\ \bullet  \tra[rrrrrrru]
&\bullet \tra[r]  & \bullet& \bullet & & } \end{xy} \end{array} = \begin{array}{c} \begin{xy}  \xymatrix@!=0.01pc{  \bullet \rtra[ddrrr]
&\bullet \rtra[rrrrd] & \bullet \ltra[r]& \bullet^p& & & & \\  & & &
& \cdots & \bullet^x& \cdots &\bullet^y \\  \bullet \ltra[rrrrrrru]&
\bullet \ltra[r]&\bullet &\bullet & & & & } \end{xy} \end{array}. $$ Starting at
the dot labelled $p$ and traversing the diagram in the direction
indicated, we see that the GCT of this diagram is $\{ ULZT \}$.
Notice that the string of letters $Z$ is unchanged as well as the
edges towards $x$ and $y$ have not changed their type nor has the
direction from which we travel from $x$ to $y$. It remains to show
that $dq=d$. Indeed, $$ dq = \begin{array}{c} 
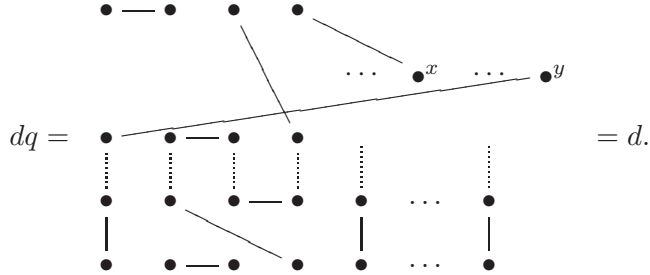
\begin{xy}
\xymatrix@!=0.01pc{ \bullet \tra[r] & \bullet & \bullet \tra[ddr] &
\bullet \tra[drr] & & \\  & & & & \cdots & \bullet^x & \cdots
&\bullet^y \\ \bullet
\tra[rrrrrrru] &\bullet \tra[r]  & \bullet& \bullet & &   & \\
\bullet \tra[d] \dta[u] &  \bullet \dta[u] \tra[drr] & \bullet
\dta[u] \ar@{-}[r] & \bullet \dta[u] &  \bullet \dta[u]  \tra[d] &
\ldots & \bullet \dta[u] \tra[d]   \\ \bullet & \bullet \tra[r]   &
\bullet    & \bullet    &  \bullet   & \ldots &  \bullet   }
\end{xy}   \end{array} = d. $$
 Notice that we can compute the cycle type just from its GCT by Remark~\ref{rk calculate CT from GCT}. 
Thus the cycle type of $\{ULTZ\}$ is just $(|-1+z|)$ where $z$ is the
difference between upward and downward pointing arrows in $Z$.
 Similarly, for $\{ ULZT\}$ we get the cycle type $(|z-1|)$ since the through arc at the end of the cycle is followed by an upper row arc and hence must be preceded by a lower row arc. Therefore the CTs are equal as claimed.
\item Now suppose $d$ has cycle type $\{ TLTZ \} $. Thus without loss of generality $d$ is of the form
$$\begin{array}{c} 
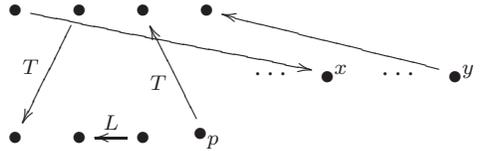
\begin{xy} \xymatrix@!=0.01pc{  \bullet \rtra[rrrrrd] &\bullet & \bullet & \bullet \ltra[rrrrd] & & & & \\ & & & & \cdots & \bullet^x & \cdots &\bullet^y \\ \bullet \ltra[uur]^T &\bullet \ltra[r]^L  & \bullet &\bullet_p \rtra[luu]^T  & & & & } \end{xy} \end{array}.$$ Here again $p$ denotes the starting point from where we read the GCT in the direction indicated. Let $q$ be $$\begin{array}{c} 
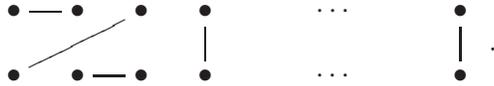
\begin{xy} \xymatrix@!=0.01pc{  \bullet \tra[r] &\bullet & \bullet \tra[lld]& \bullet \tra[d] & &\cdots & & \bullet \tra[d] \\  \bullet &\bullet \tra[r]  & \bullet &\bullet  & & \cdots & & \bullet } \end{xy} \end{array}.  $$ Then $qd$ is given by $$\begin{array}{c} 
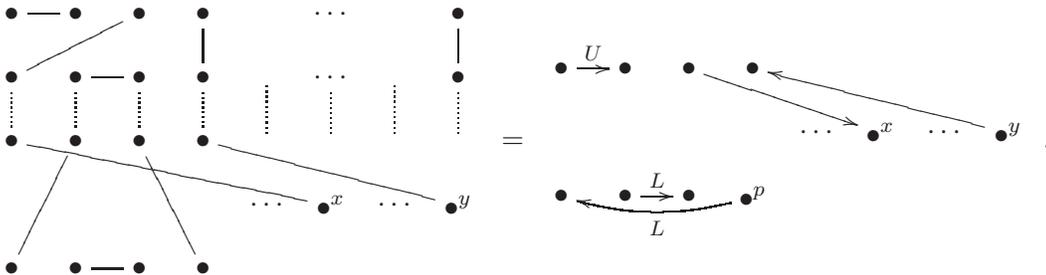
\begin{xy} \xymatrix@!=0.01pc{  \bullet \tra[r] &\bullet & \bullet \tra[lld]& \bullet \tra[d] & &\cdots & & \bullet \tra[d] \\  \bullet &\bullet \tra[r]  & \bullet &\bullet  & & \cdots & & \bullet \\ \bullet \tra[rrrrrd] \dta[u] &\bullet \dta[u] & \bullet \dta[u] & \bullet \tra[rrrrd] \dta[u] & \dta[u]&\dta[u] & \dta[u]& \dta[u] \\ & & & & \cdots & \bullet^x & \cdots &\bullet^y \\ \bullet \tra[uur] &\bullet \tra[r]  & \bullet &\bullet \tra[luu]  & & & & } \end{xy} \end{array}   = \begin{array}{c} \begin{xy} \xymatrix@!=0.01pc{  \bullet \rtra[r]^U &\bullet & \bullet \rtra[rrrd] & \bullet \ltra[rrrrd] & & & & \\ & & & & \cdots & \bullet^x & \cdots &\bullet^y \\ \bullet \larc[rrr]_L &\bullet \rtra[r]^L  & \bullet &\bullet^p  & & & & } \end{xy} \end{array}.  $$ Starting at $p$ and continuing in the direction indicated, we can see that this diagram has cycle type $\{ LULZ\}$ since the edges ending in $x$ and $y$ are still attached to the top row. Hence their type has not changed and we also have the same orientation for $Z$. Again it remains to show that $dq=d$ which is easily verified:
$$\begin{array}{c} 
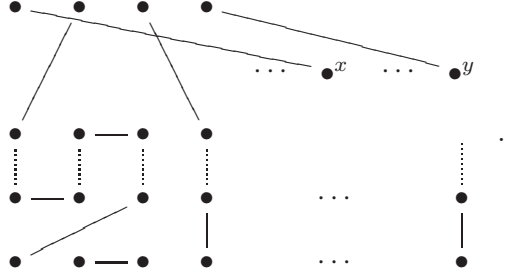
\begin{xy} \xymatrix@!=0.01pc{  \bullet \tra[rrrrrd] &\bullet & \bullet & \bullet \tra[rrrrd] & & & & \\ & & & & \cdots & \bullet^x & \cdots &\bullet^y \\ \bullet \tra[uur] &\bullet \tra[r]  & \bullet &\bullet \tra[luu]  & & & &  \\     \bullet \dta[u] \tra[r] &\bullet  \dta[u] & \bullet \dta[u] \tra[lld]& \bullet  \dta[u] \tra[d] & &\cdots & & \bullet  \dta[u] \tra[d] \\  \bullet &\bullet \tra[r]  & \bullet &\bullet  & & \cdots & & \bullet   } \end{xy} \end{array}.$$

 Finally, notice that the CT of $\{TLTZ\}$ is just $(|1-1+z|)=(|z|)$ and the CT of $\{ LULZ\}$ is $(|z|)$ as well.

\item Suppose $d$ has cycle type $\{ Z,ULZ'\}$. Hence without loss of generality $d$ is of the form $$\begin{array}{c} 
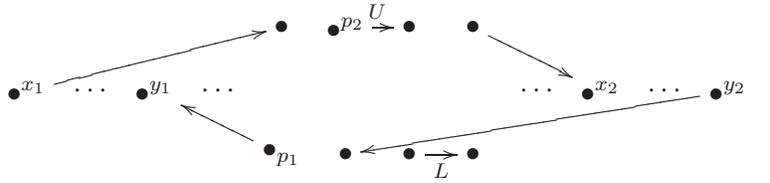
\begin{xy} \xymatrix@!=0.01pc{  & & & & \bullet \ltra[lllld] &\bullet^{p_2} \rtra[r]^U & \bullet & \bullet \rtra[rrd] & & & & \\ \bullet^{x_1} & \cdots & \bullet^{y_1} & \cdots &  & & & & \cdots & \bullet^{x_2} & \cdots &\bullet^{y_2} \\ & & & & \bullet_{p_1} \rtra[llu] & \bullet \ltra[rrrrrru] &\bullet \rtra[r]_L  & \bullet  & & & & } \end{xy} \end{array}.$$ Here we have as usual drawn the not necessarily distinct dots $x_1,y_1,x_2,y_2$ in the middle to signify that no assumption has been made about them being in the top or bottom row nor about whether they are on the left or right of the explicitly drawn dots. The GCT of the leftmost component is read starting at $p_1$ in the direction indicated and corresponds to the string $Z$. In the other component corresponding to $Z'$ we start with dot $p_2$ and continue in the direction indicated. Let $q$ be $$\begin{array}{c} 
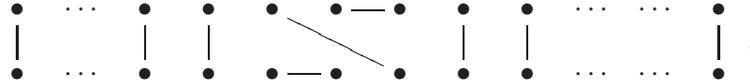
\begin{xy} \xymatrix@!=0.01pc{  \bullet \tra[d] & \cdots & \bullet \tra[d] & \bullet \tra[d] & \bullet \tra[rrd] &\bullet \tra[r] & \bullet  & \bullet \tra[d] & \bullet \tra[d] & \cdots & \cdots & \bullet \tra[d]  \\ \bullet & \cdots & \bullet & \bullet & \bullet \tra[r] & \bullet &\bullet  & \bullet  & \bullet & \cdots & \cdots & \bullet } \end{xy} \end{array}. $$
Then the product $dq$ is just
$$\begin{array}{c} 
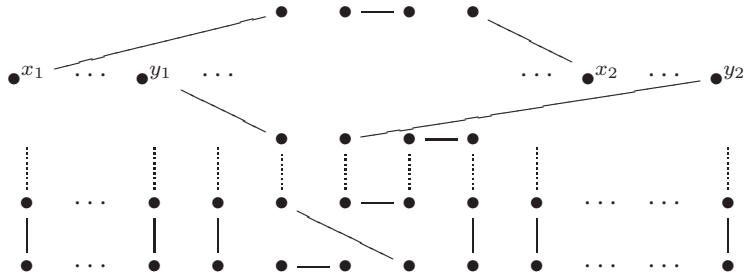
\begin{xy} \xymatrix@!=0.01pc{  & & & & \bullet \tra[lllld] &\bullet \tra[r] & \bullet & \bullet \tra[rrd] & & & & \\ \bullet^{x_1} & \cdots & \bullet^{y_1} & \cdots &  & & & & \cdots & \bullet^{x_2} & \cdots &\bullet^{y_2} \\ & & & & \bullet \tra[llu] & \bullet \tra[rrrrrru] &\bullet \tra[r]  & \bullet  & & & & \\  \bullet \tra[d] \dta[u] & \cdots & \bullet \dta[u]  \tra[d] & \bullet \dta[u]  \tra[d] & \bullet  \dta[u]  \tra[rrd] &\bullet  \dta[u]  \tra[r] & \bullet  \dta[u]  & \bullet  \dta[u]  \tra[d] & \bullet  \dta[u]  \tra[d] & \cdots & \cdots & \bullet \dta[u]  \tra[d]  \\ \bullet & \cdots & \bullet & \bullet & \bullet \tra[r] & \bullet &\bullet  & \bullet  & \bullet & \cdots & \cdots & \bullet  } \end{xy} \end{array} $$ which can be seen to be equal to
$$\begin{array}{c} 
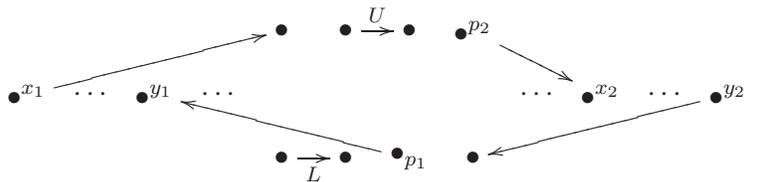
\begin{xy} \xymatrix@!=0.01pc{  & & & & \bullet \ltra[lllld] &\bullet \rtra[r]^U & \bullet & \bullet^{p_2} \rtra[rrd] & & & & \\ \bullet^{x_1} & \cdots & \bullet^{y_1} & \cdots &  & & & & \cdots & \bullet^{x_2} & \cdots &\bullet^{y_2} \\ & & & & \bullet \rtra[r]_L &\bullet  & \bullet_{p_1} \rtra[llllu] &\bullet \ltra[rrrru]  & & & & } \end{xy} \end{array}. $$ Starting at $p_1$ and $p_2$, respectively, and continuing in the direction indicated, it follows that the cycle type of $dq$  is $\{ ZLU, Z'\}$ since the dots connected to $x_1,y_1,x_2$ and $y_2$ are still in the same row as before and strings $Z$ and $Z'$ are still traversed in the same direction. On the other hand $qd$ is given by $$\begin{array}{c} 
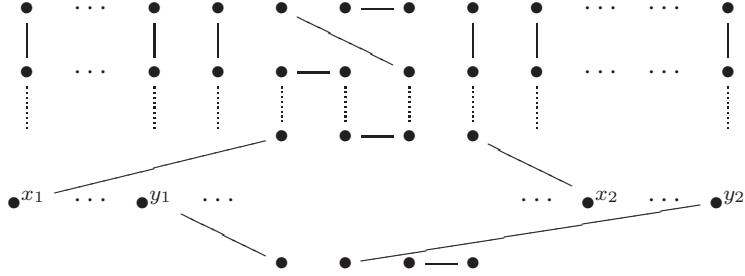
\begin{xy} \xymatrix@!=0.01pc{   \bullet \tra[d] & \cdots & \bullet \tra[d] & \bullet \tra[d] & \bullet \tra[rrd] &\bullet \tra[r] & \bullet  & \bullet \tra[d] & \bullet \tra[d] & \cdots & \cdots & \bullet \tra[d]  \\ \bullet \dta[d]  & \cdots & \bullet  \dta[d] & \bullet \dta[d]  & \bullet \tra[r] \dta[d]  & \bullet  \dta[d] &\bullet   \dta[d] & \bullet  \dta[d]  & \bullet \dta[d]  & \cdots & \cdots & \bullet \dta[d]     \\ & & & & \bullet \tra[lllld] &\bullet \tra[r] & \bullet & \bullet \tra[rrd] & & & & \\ \bullet^{x_1} & \cdots & \bullet^{y_1} & \cdots &  & & & & \cdots & \bullet^{x_2} & \cdots &\bullet^{y_2} \\ & & & & \bullet \tra[llu] & \bullet \tra[rrrrrru] &\bullet \tra[r]  & \bullet  & & & & } \end{xy} \end{array} $$ which is just equal to $d$. To complete the proof notice that if the differences of upward and downward pointing arrows in $Z$ and $Z'$ are $z$ and $z'$, respectively, then the CT of $\{ Z, ULZ'\}$ is $(|z|,|z'|)$ and the CT of $\{ ZLU, Z'\}$ is again $(|z|,|z'|)$.
\end{enumerate}

\end{proof}
We are now in a position to prove the following lemma, from which
Theorem~\ref{theo ct determined uv conjugation} will follow as a
corollary:

\begin{prop}
Let $d \in \brd$ be a diagram with cycle type
$\mu=(\mu_1,\mu_2,\ldots,\mu_k,0^m)$ where $\mu_i \neq 0$ for $i=1,\ldots,k$. Then the GCT of $d$ is
$\sim$-equivalent to \begin{align} \label{equation desired gct} \{
(UL)^{r-|\mu|}T^{\mu_1},T^{\mu_2},T^{\mu_3}, \ldots,
T^{\mu_k},UL,UL,\ldots,UL \} \end{align} where $UL$ appears $m$
times at the end.
\end{prop}

\begin{proof}
Notice that we may assume that $d$ is in cycle form, as defined in Definition~\ref{defn cycle}. Otherwise, conjugate $d$ by an appropriate permutation. We will show that the GCT of $d$ is equivalent to the string above in two steps:

\subsubsection*{First step:}
In the first step we will simplify the GCT of each cycle of $d$.  If
the cycle just consists of through arcs then its GCT is of the form
$\{T^l \}$ for some natural number $l$. Thus the GCT of the cycle is
consistent with the GCT given in Equation (\ref{equation desired gct}).
Therefore, we only have to consider cycles with at least one
horizontal arc. A general GCT of such a cycle is $$a=\{
UT^{x_1}LT^{x_2}U \ldots L T^{x_k} \} ,$$ where the variables $x_i$
are non-negative integers. Here we have assumed that the first
letter is a $U$ which is justified by Part~(\ref{gctcalc1}) of Lemma
\ref{lem gct calculus}. We shall be making use of this fact
frequently.  Now the claimed form follows since the letters $U$ and $L$ must
always alternate by Remark~\ref{rk Us and Ls alternate}. Our aim is
to separate the strings $UL$ and $T$ in the GCT of each cycle so that each
cycle has a GCT of the form $(UL)^kT^m$ for natural numbers $m$ and
$k$. We will start by showing that there is an equivalent GCT in
which no $L$ is preceded by a $T$, that is, the string of letters $TL$
does not occur.

Suppose this string of letters does occur in $a$. Then there are two
possibilities: Either $TL$ is followed by a through arc in which
case we get $a=\{ TLTZ \}$ for some string of letters $Z$, or $TL$
is followed by a horizontal arc in which case $a=\{ TLUZ\}$ since
again the letters $U$ and $L$ must alternate by Remark~\ref{rk Us and Ls
alternate}.

In the first case we know that by Part~(\ref{gctcalc3}) of Lemma
\ref{lem gct calculus} we have $a \sim \{ LULZ \} $. Since all other
letters except $TLT$, that is $Z$, are unchanged, we can repeat this
with every occurrence of the string $TLT$  without affecting any
other letters until the string $TLT$ does not occur anymore.

In the second case use parts~(\ref{gctcalc1}) and (\ref{gctcalc2})
of Lemma~\ref{lem gct calculus} to get the equivalence
\begin{align} \{ TLUZ \} \sim \{ ZTLU \} \sim \{ ULT \rho(Z) \} \sim
\{ UL\rho(Z)T \} \sim \{TZLU \} \sim \{ LUTZ  \} \label{TLUZ sim
LUTZ} \end{align} in other words we can move the letter $T$ past a
string $LU$ and still get an equivalent GCT. We will use this rule
to decrease the number of letters $L$ which are preceded by a $T$. Notice
that simply using $\{ TLUZ \} \sim \{ LUTZ \}$ will not guarantee
that the number of letters $L$ preceded by a $T$ has decreased since the
first letter in the string $Z$ might be an $L$. We thus have to make
more explicit assumptions about $Z$. We will distinguish the
following three cases:
\begin{enumerate}[(I)]
\item \emph{Assumptions: There is more than one $L$ preceded by a $T$ and the leftmost $T$ occurring in $Z$ is preceded by a $U$.} Consider what these assumptions mean for $Z$. The first assumption implies that $Z$ must have at least one substring of the form $T^kL$ for some $k \neq 0$. The second assumption implies that the leftmost string of the form $T^kL$ is either at the beginning of $Z$ or preceded by a $U$. However, $Z$ cannot begin with a string of the form $UT^kL$ as this would imply that the cycle type $a=\{TLUZ\}$ would have two consecutive upper row horizontal arcs, which is impossible. Thus the substring must be of the form $LUT^kL$. Again if this substring is not at the beginning of $Z$, it must be preceded by a $U$ since by assumption we had reached the leftmost $T$ already. Repeating this argument, we can assume that $Z=(LU)^mT^kLZ'$ where $m$ might be zero. Thus using (\ref{TLUZ sim LUTZ}) repeatedly yields $$a = \{ TLU\left[ (LU)^mT^kLZ' \right] \}  \sim \{  LUT(LU)^mT^kLZ' \}  \sim  \ldots \sim \{  (LU)^{m+1}T^{k+1}LZ' \}  $$ and we have decreased the number of letters $L$ preceded by a $T$ by $1$.
\item \emph{Assumption: There is exactly one $L$ preceded by a $T$ and the leftmost $T$ occurring in $Z$ is preceded by a $U$.} Thus $Z=(LU)^mT^k$ and no $L$ follows as this would result in a second $L$ preceded by a $T$. We can easily see that by repeated application of (\ref{TLUZ sim LUTZ}) we get $$ TLU(LU)^mT^k \sim (UL)^{m+1}T^{k+1}$$ so that there is no $L$ preceded by a $T$ left.
\item \emph{Assumption: The leftmost $T$ occurring in $Z$ is preceded by an $L$.} We can deduce that $Z=(LU)^mLT^kUZ'$  where $k \geq 1$. Using (\ref{TLUZ sim LUTZ}) we get $$\{ TLU(LU)^mLT^kUZ' \} \sim \{ (LU)^{m+1}TLTT^{k-1}UZ' \} .$$ Applying parts~(\ref{gctcalc1}) (shifting) and (\ref{gctcalc3}) of Lemma~\ref{lem gct calculus} yields $$\{ (LU)^{m+1}TLTT^{k-1}UZ' \} \sim \{ (LU)^{m+1}LULT^{k-1}UZ' \}, $$ and we have again decreased the number of letters $L$ preceded by a $T$ by $1$.
\end{enumerate}
Thus we can repeat the above until there is no $L$ preceded by a $T$
left. This means that our original GCT is equivalent to
$\{ULT^{m_1}ULT^{m_1}UL \ldots ULT^{m_k}\}$ for some non-negative
integers $m_i$. This follows since the letters $U$ and $L$ must
alternate but there cannot be a $T$ between a letter $U$ and a
letter $L$ as this would result in an $L$ preceded by a $T$. But by
Part~(\ref{gctcalc2}) of Lemma~\ref{lem gct calculus} we have that
$ \{ ULTZ \} \sim \{ ULZT \}$, and thus we can move all the letters
$T$ to the end and our GCT is equivalent to $\{ (UL)^mT^k\}$ for
some $k$.

\subsubsection*{Second step:}
Notice that every cycle with no through arcs must be of the form
$(UL)^x$ for some natural number $x$. Since we can conjugate by a
permutation to rewrite the diagram in cycle form and then apply the
operations in the first step to each of the cycles which contain
horizontal and through arcs independently, we can assume that our
diagram has GCT equal to $\{
(UL)^{m_1}T^{k_1},(UL)^{m_2}T^{k_2},\ldots,(UL)^{m_l}T^{k_l} \}$ for
non-negative integers $m_i,k_i$. We can suppose that this GCT is
ordered in the following way: All components with $m_i=1$ and
$k_i=0$, that is all $UL$ components, and all components with $m_i=0$, that is all pure through string components, are at the end of the GCT. Furthermore, the remaining components are ordered in such a way that $m_i
\geq m_j$ whenever $i \leq j$. Distinguish the following cases:
\begin{enumerate}[(I)]
\item If any but the first cycle has horizontal and through arcs, that is $k_i,m_i >0$ for $i \geq 2$, then we can use Part~(\ref{gctcalc4}) of Lemma~\ref{lem gct calculus} to bring all the $UL$ strings from the corresponding component to the first component (as certainly $m_1 \geq m_i>0$) so that the component only consist of through strings and the GCT is still equivalent to the original one.
\item Similarly, if $m_i \geq 2$ and $k_i = 0$ for $i \geq 2$ then we can again use Part~(\ref{gctcalc4}) of Lemma~\ref{lem gct calculus} to move all $UL$ strings to the first component until the original component is equal to $UL$.
\item If any but the first component are either of the form $UL$ or $T^k$ for some non-negative integer $k$, we already have a GCT of the desired form.
\end{enumerate}

Thus we are left with only one component containing both horizontal
and vertical arcs, namely, the first one, and all others are either
of the form $UL$ or consist of through arcs only. Since this GCT is
still equivalent to the GCT of the diagram $d$ we started with, we
know that the CT of this diagram is
$\mu=(\mu_1,\mu_2,\ldots,\mu_k,0^m)$ by Definition~\ref{defn of gct
equiv rel}. Thus the only possibility is that the new GCT is
precisely $$\{ (UL)^{r-|\mu|}T^{\mu_1},T^{\mu_2},T^{\mu_3}, \ldots,
T^{\mu_k},UL,UL,\ldots,UL \}$$ where $UL$ occurs $m$ times. The
result follows.
\end{proof}

\subsection{A necessary condition for $(uv,vu)$-conjugacy in $H$}
We now return to the general case of studying $(uv,vu)$-conjugation in $H$. We will start by extending the notion of cycle type from $D$ to $H$:

\begin{defn}
Let $h \in H$ and suppose $h=\delta^k d$ where $k \in \N_0$ and $d \in D$. Suppose $d$ has cycle type $(\alpha)$ where $\alpha$ is a list of non-negative integers. Define the cycle type of $h$ to be the partition $(\alpha,0^k)$.
 \end{defn}

The following theorem will be crucial for relating cycle types which only differ in the number of $0$s:

\begin{thm}
\label{thm char of arbitrary diagram} Let $\rho: \brd \to M_n(F)$ be
any $\brd$-representation,  $d \in \brd$ be a diagram of cycle type
$\mu=(\alpha,0)$ and $d' \in \brd$ be a diagram of cycle type  $(\alpha)$ where $\alpha$ is a list of non-negative integers. Then $d \sim \delta d'$.

\end{thm}

\begin{proof}
Let $d$ have cycle type $(\alpha,0)$. By Theorem~\ref{theo ct determined uv conjugation}, we may assume that $d$ is of the form

$$
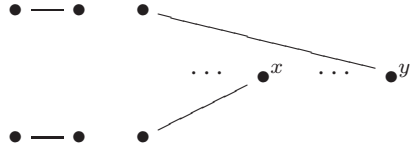
\begin{xy}
\xymatrix@!=0.01pc{    \bullet \tra[r] &  \bullet &  \bullet \tra[rrrrd] & & & & \\  & & & \cdots & \bullet^x & \cdots & \bullet^y \\  \bullet \tra[r] &  \bullet &  \bullet \tra[urr] & & & &   }
\end{xy}$$ where we have again only drawn the first three columns explicitly and assumed that the dots connected to the top  and bottom  dots of column $3$ are $y$ and $x$, respectively. Notice that $x$ and $y$ have been drawn in the middle to demonstrate that no assumption has been made over their exact position. Let $q$ be the diagram $$ \begin{array}{c} 
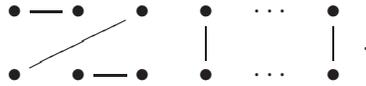
\begin{xy} \xymatrix@!=0.01pc{ \bullet \tra[r] & \bullet  &  \bullet \tra[lld] & \bullet \tra[d]  &  \cdots & \bullet \tra[d]  \\ \bullet & \bullet \tra[r]  &  \bullet & \bullet  &  \cdots & \bullet} \end{xy} \end{array}. $$ Thus $qd$ is given by $$
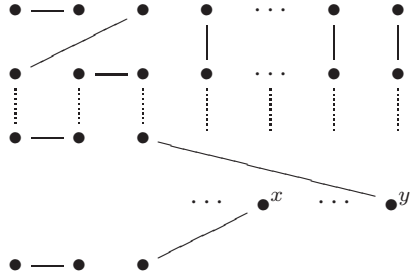
\begin{xy}
\xymatrix@!=0.01pc{  \bullet \tra[r] & \bullet  &  \bullet \tra[lld] & \bullet \tra[d]  &  \cdots & \bullet \tra[d] & \bullet \tra[d]  \\ \bullet \dta[d] & \bullet \dta[d] \tra[r]  &  \bullet  \dta[d] & \bullet   \dta[d] &  \cdots  \dta[d] & \bullet  \dta[d] & \bullet \dta[d]   \\ \bullet \tra[r] &  \bullet &  \bullet \tra[rrrrd] & & & & \\  & & & \cdots & \bullet^x & \cdots & \bullet^y \\  \bullet \tra[r] &  \bullet &  \bullet \tra[urr] & & & &   }
\end{xy}$$
which is just equal to $d$. On the other hand $d'=dq$ is given by

$$
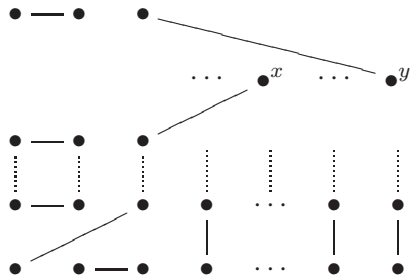
\begin{xy}
\xymatrix@!=0.01pc{    \bullet \tra[r] &  \bullet &  \bullet \tra[rrrrd] & & & & \\  & & & \cdots & \bullet^x & \cdots & \bullet^y \\  \bullet \tra[r] &  \bullet &  \bullet \tra[urr] & & & &  \\  \bullet \tra[r] \dta[u] & \bullet   \dta[u] &  \bullet \tra[lld]  \dta[u] & \bullet \tra[d]  \dta[u]  &  \cdots  \dta[u] & \bullet \tra[d] \dta[u] & \bullet  \dta[u]  \tra[d]  \\ \bullet & \bullet \tra[r]  &  \bullet & \bullet  &  \cdots & \bullet & \bullet    }
\end{xy}$$ which can be seen to be equal to $\delta$ times the diagram $$d''=  \begin{array}{c} 
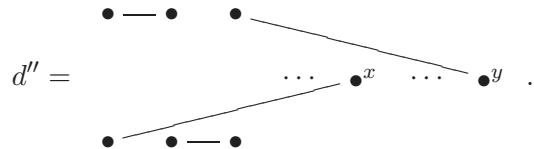
\begin{xy}
\xymatrix@!=0.01pc{    \bullet \tra[r] &  \bullet &  \bullet \tra[rrrrd] & & & & \\  & & & \cdots & \bullet^x & \cdots & \bullet^y \\  \bullet \tra[rrrru] &  \bullet \tra[r] &  \bullet& & & &   }
\end{xy} \end{array}.$$ Notice that $d''$ has cycle type $(\alpha)$ and hence Theorem~\ref{theo ct determined uv conjugation} implies that $d'' \sim d'$. Thus, we have found a diagram $q$ such that $d=dq$ and $\delta d'' = qd$ so that we can deduce $d \sim \delta d'' \sim \delta d'$, as required.
\end{proof}

\section{Eigenvalues of matrix representations}
We will now take a slight detour and study eigenvalues and characters of matrix representations. The reason for studying character theory first is that it will give us a convenient tool to completely determine when two diagrams are  $\chi$-conjugate. 

Notice that over fields such that the Brauer algebras is semisimple this could also be deduced from work of Ram, see \cite{Ram}. However, we would like to show the result in full generality and therefore have to extend his work using  different methods.

Let us record a direct consequence of Theorem~\ref{theo ct determined uv conjugation}. This will relate the
cycle type to the minimal polynomial and to geometric multiplicities of eigenvalues of certain matrices:

\begin{theo}
\label{theo same ct implies same eigenvalues}
Let $\rho: \brd \to M_n(F)$ be a matrix representation of $\brd$ and suppose $d,d' \in D$ are elements with the same cycle type. Then:
\begin{enumerate}[(i)]
\item The minimal polynomials of $\rho(d)$ and  $\rho(d')$ only differ by a power of the indeterminate, and in particular,  $0 \neq \lambda \in F$ is an eigenvalue of $\rho(d)$ if and only if $\lambda$ is an eigenvalue of $\rho(d')$.
\item For any $0 \neq \lambda \in F$ the geometric multiplicity of $\lambda$ as an eigenvalue of $\rho(d)$ and  $\rho(d')$, respectively, is the same. That is,  if $V_\lambda$ and $V_{\lambda}'$ are the eigenspaces of $\lambda$ corresponding to $\rho(d)$ and $\rho(d')$, respectively, then $\dim_F V_\lambda= \dim_F V_{\lambda}'$.
\end{enumerate}
\end{theo}

\begin{proof}
By Theorem~\ref{theo ct determined uv conjugation}, we know that $d_1 \sim d_2$. This will suffice to imply the result. Start with the case that $d \sim_q d'$. Thus either $d$ and $d'$ have the same generalized cycle type or there is some $q \in H$ such that $d=qd$ and $d'=dq$. In the first case both statements follow from Theorem~\ref{theorem sigma conjugate iff gct the same}. So consider the second case:
\begin{enumerate}[(i)]
\item Let us first show that for any matrix representation $\rho: \brd
\to M_n(F)$ we have that the minimal polynomials of $\rho(d)$ and
$\rho(d')$ only differ by a power of the indeterminate.

Let $f(x) = \sum_{i=0}^s \alpha_i x^i \in F[x]$ for some $\alpha_i
\in F$ and $s \in \N$ be the minimal polynomial of $\rho(d)$, so
that $f(x)$ is the monic polynomial of minimal degree satisfying
$f(\rho(d))=0$. Similarly, let $g(x)=\sum_{i=0}^t \beta_i x^i \in
F[x]$ for some $\beta_i \in F$ and $t \in \N$ be the minimal
polynomial of $\rho(d')$. If $\alpha_0 \neq 0$ then  $0$ is not an
eigenvalue of $\rho(d)$. This implies that $\rho(d)$ is invertible.
It follows that $\rho(q) = I_n$, the $n\times n$ identity matrix, as
$\rho(d)=\rho(q)\rho(d)$. Hence $\rho(d') = \rho(d) \rho(q) =
\rho(d)$ and in particular $\rho(d)$ and $\rho(d')$ have the same
eigenvalues so that the result follows. Similarly, if $\beta_0 \neq
0$ the result follows. So it remains to consider the case
$\alpha_0=\beta_0 = 0$.

Now observe that for any $1 \leq l \in \N$ we have \begin{align}
(d')^l=(dq)^l=d(qd)(qd)q\ldots d(qd)q=d^lq. \label{eqn d prime to
the l} \end{align} Similarly, \begin{align}
d^l=(qd)^l=q(dq)(dq)d\ldots q(dq)d=q(d')^{l-1}d. \label{eqn d to the
l} \end{align}

Since $\alpha_0 =0$, we can use Equation~(\ref{eqn d prime to the
l}) to obtain
$$f(\rho(d'))= \sum_{i=1}^s \alpha_i (\rho(d'))^i = (\sum_{i=1}^s
\alpha_i \rho(d)^i)\rho(q) = f(\rho(d))\rho(q)=0,$$ so that
$\rho(d')$ satisfies the minimal polynomial of $\rho(d)$. We can
deduce that the minimal polynomial $g(x)$ of $\rho(d')$ must divide
$f(x)$, that is, $f(x)=g(x)h(x)$ for some $h(x) \in F[x]$.

On the other hand using Equation~(\ref{eqn d to the l}) and $\beta_0
=0$, we obtain
$$\rho(d)g(\rho(d))=\rho(d) \sum_{i=1}^t \beta_i \rho(d)^i
=\rho(d)\rho(q) (\sum_{i=1}^t \beta_i \rho(d')^{i-1})\rho(d) =
g(\rho(d'))\rho(d)=0.$$ This time we deduce that $f(x)=g(x)h(x)$
divides $xg(x)$. It follows that either $h(x)=x$ or $h(x)=1$ and in
particular $f(x)=xg(x)$ or $f(x)=g(x)$, as required.

Notice that this implies that $0 \neq \lambda \in F$ is an eigenvalue of $\rho(d)$ if and only if it is an eigenvalue of $\rho(d')$.

\item It remains to show that, given $\lambda \neq 0$ a common
eigenvalue of  $\rho(d)$ and $\rho(d')$ with corresponding
eigenspaces $V_\lambda$ and $V_{\lambda}'$,  then $\dim V_\lambda=
\dim V_{\lambda}'$. In fact we will show that $V_\lambda=
V_{\lambda}'$.

If $v \in V_{\lambda}$ then $\rho(d)v=\lambda v$. By assumption
$d=qd$. Hence $\lambda v =(\rho(q)\rho(d))v=\lambda \rho(q)v$, so
that $\rho(q)v =v$. Thus $\rho(d')v = \rho(d) \rho(q)v = \rho(d) v =
\lambda v$, that is $v \in V_{\lambda}'$.

Conversely if $v \in V_{\lambda}'$ then $\rho(d')v = \lambda v$.
Using $qd'=qdq=dq=d'$ we get $\lambda v =
\rho(d')v=\rho(q)\rho(d')v=\lambda \rho(q)v$, which again implies
$\rho(q)v=v$ and hence $\rho(d)v = \rho(d) \left[ \rho(q)v \right]
=\rho(d')v=\lambda v$. Therefore $v \in V_{\lambda}$, as required. \end{enumerate}

We have thus established the claim in the case $d \sim_q d'$. But both statements are transitive, that is, if they hold for $d_1$ and $d_2$ and also for $d_2$ and $d_3$ then they will also be true for $d_1$ and $d_3$. Thus, the result extends to the transitive closure of $\sim_q$ which is precisely $\sim$.
\end{proof}

We will shortly show that the characteristic polynomials of $\rho(d_1)$ and
$\rho(d_2)$ are even equal.
However, the minimal polynomials will not be the same in general. If
we consider the diagrams $$ d_1 =
\begin{array}{c} \begin{xy}
 \xymatrix@!=0.01pc{ \bullet \tra[r] & \bullet  & \bullet
\tra[ld] & \bullet \tra[ld]  & \bullet \tra[lllld]  \\ \bullet &
\bullet  & \bullet  & \bullet\tra[r]  & \bullet  }
\end{xy} \end{array}, \qquad d_2= \begin{array}{c}  \begin{xy} \xymatrix@!=0.01pc{ \bullet
\tra[r] &
\bullet & \bullet \tra[r] & \bullet  & \bullet \tra[lllld]  \\
\bullet & \bullet \tra[r] & \bullet  & \bullet  \tra[r]& \bullet  }
\end{xy} \end{array},$$
 then both have cycle type $(1)$. But one verifies that
$d_1^2=d_1$ while $d_2^2 \neq d_2$ and thus their minimal
polynomials in the regular representation are already not the same.

To study eigenvalues associated to diagrams with a given cycle type, we will introduce a canonical diagram for
each cycle type. The multiplication of canonical diagrams will be well-behaved and will allow us to study eigenvalues. We need the following definition:

\begin{defn}
\begin{enumerate}[(i)]
\item Let $a \in B_r (\delta)$ and $b \in B_s (\delta)$ be Brauer diagrams. Define $a \otimes b \in B_{r+s}(\delta)$ to be the diagram obtained from $a$ and $b$ by placing them next to each other. This notation is not to be confused with the tensor product of modules and it will be clear from the context which of the two is being used.
\item Define $e$ to be the diagram consisting of one horizontal arc in the top and bottom row:
$$e=\begin{array}{c} \begin{xy} \xymatrix@!=0.01pc{ \bullet \ar@{-}[r] & \bullet \\ \bullet \ar@{-}[r]  & \bullet} \end{xy} \end{array}.$$
\item For any $k \in \N$, define $\gamma_k$  to be the permutation diagram of $(1,2,\ldots,k)$. Given a partition $\lambda=(\lambda_1,\lambda_2,\ldots,\lambda_n)$, set $\gamma_{\lambda}= \gamma_{\lambda_1} \otimes  \gamma_{\lambda_2} \otimes \ldots \otimes \gamma_{\lambda_n}.$
\item For any $t \in \N_0$ such that $r-2t \geq 0$, define a map $h_t: \Sigma_{r-2t} \to \brd$ in the following way: For any $\sigma \in \Sigma_{r-2t}$, let $h_t(\sigma)$ be the diagram on $r$ dots with top row horizontal arcs from dots $1$ to $2$, $3$ to $4$, $\ldots$, $2t-1,2t$ and bottom row horizontal arcs from dots $2$ to $3$, $4$ to $5$, $\ldots$,
$(2t),(2t+1)$,
that is
 with the following arc structure $$h_t(\sigma) = \begin{array}{c} \begin{xy}
\xymatrix@!=0.01pc{ \bullet^1 \tra[r] &  \bullet^2 & \cdots  & \bullet \tra[r]  &  \bullet^{2t}  &  \bullet   & \bullet    & \cdots   & \bullet^r    \\ \bullet^{1}  & \bullet^{2} \tra[r] &  \bullet^{3} & \cdots & \bullet^{2t}  \tra[r] & \bullet  &  \bullet  & \cdots   & \bullet^{r}   }
\end{xy} \end{array}.$$
The remaining dots form the permutation $\sigma$ in the following way: For $i = 1,\ldots,r-2t$,
dot $2t+i$ in the top row is connected to dot $2t+\sigma(i)$ in the bottom row unless $\sigma(i) =1$ in which case dot $2t+i$ in the top row will be connected to dot $1$ in the bottom row.
\end{enumerate}
\end{defn}

Let us illustrate these definitions at some examples:
\begin{ex}
\begin{enumerate}[(i)]
\item The diagram $e \otimes e \otimes \gamma_{(3,2)}= e^{\otimes 2} \otimes \gamma_3 \otimes \gamma_2$
is given by
$$\begin{array}{c} 
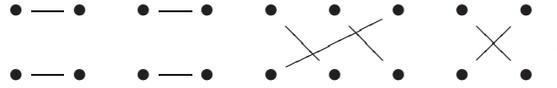
\begin{xy}
\xymatrix@!=0.01pc{   \bullet \ar@{-}[r]  & \bullet  & \bullet
\ar@{-}[r]  & \bullet & \bullet \ar@{-}[rd]   & \bullet \ar@{-}[rd]
& \bullet \tra[lld] & \bullet \tra[rd] & \bullet \tra[ld] \\ \bullet
\ar@{-}[r]  & \bullet  & \bullet \ar@{-}[r]  & \bullet & \bullet  &
\bullet & \bullet & \bullet & \bullet}

\end{xy}\end{array}.
$$
\item The diagram $h_2((1,2,3))$ is equal to $$\begin{array}{c} 
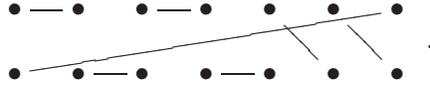
\begin{xy}
\xymatrix@!=0.01pc{ \bullet \tra[r] &  \bullet & \bullet \tra[r]  &
\bullet  &  \bullet \tra[rd] &  \bullet \tra[rd]  & \bullet
\tra[lllllld]  \\ \bullet  & \bullet \tra[r] &  \bullet & \bullet
\tra[r] & \bullet  & \bullet  &  \bullet }
\end{xy}\end{array}.$$

\end{enumerate}
\end{ex}

The next lemma shows that, with a suitable restriction of the codomain, the map $h_t$  is a homomorphism of groups:

\begin{lem}
\label{lem properties of ht} For any $t \in \N_0$ such that $r-2t
\geq 0$  the map $h_t$ is injective, and for any $\sigma_1,\sigma_2
\in \Sigma_{r-2t}$ we have $$
h_t(\sigma_1\sigma_2)=h_t(\sigma_1)h_t(\sigma_2).$$
\end{lem}

\begin{proof}
It is clear from the definition that $h_t$ must be injective.
Multiplying $h_t(\sigma_1)$ with $h_t(\sigma_2)$, we obtain
$$\begin{array}{c} 
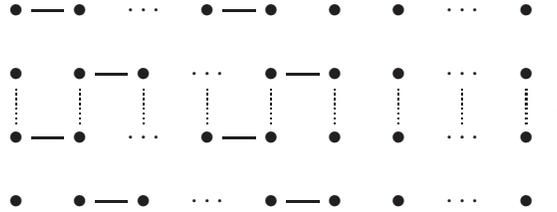
\begin{xy} \xymatrix@!=0.01pc{ \bullet \tra[r] &  \bullet & \cdots
& \bullet \tra[r]  &  \bullet  &  \bullet   & \bullet    & \cdots &
\bullet    \\ \bullet  & \bullet \tra[r] &  \bullet & \cdots &
\bullet  \tra[r] & \bullet  &  \bullet  & \cdots   & \bullet \\
\bullet \tra[r] \dta[u] &  \bullet \dta[u]  & \cdots   \dta[u] &
\bullet \tra[r]   \dta[u] &  \bullet  \dta[u]  &  \bullet  \dta[u]
& \bullet   \dta[u]   & \cdots    \dta[u] & \bullet  \dta[u]    \\
\bullet  & \bullet \tra[r] &  \bullet & \cdots & \bullet  \tra[r] &
\bullet  &  \bullet  & \cdots   & \bullet  }
\end{xy} \end{array}.$$ We see that this product will simply be $h_t(\sigma_1 \sigma_2)$ since dot $1$ in the second row of the diagram above will be identified with dot $2t+1$ in the third row, and for $1 \neq i$ dot $2t+i$ in the second row will be identified with dot $2t+i$ in the third row.
\end{proof}

We can now define canonical forms in $\brd$ for diagrams of cycle type $\mu$:

\begin{defn}
\label{defn canonical diagram of cycle type}
Let $\mu$ be a cycle type for $\brd$. If
$\mu=(\mu_1,\mu_2,\ldots,\mu_l,0^{m_0})$ where $\mu_i \neq 0$ for $i=1,\ldots,k$ and $m_0 \in \N_0$, then define the
canonical diagram $d_\mu$ as follows: If
$\mu'=(\mu_1,\mu_2,\ldots,\mu_l)$ is the partition formed by the
non-zero parts of $\mu$ and $t= \frac{r-|\mu|}{2}-m_0$ then $$
d_{\mu}=h_t(\gamma_{\mu_1}) \otimes \gamma_{\mu' \smallsetminus
\mu_1} \otimes e^{\otimes m_0}$$ where $\mu' \smallsetminus \mu_1 = (
\mu_2,\ldots,\mu_l)$.
\end{defn}

\begin{rk}
\label{rk which ct do occur}
The diagram $d_{\mu}$ has cycle type $\mu$ as is easily verified.

\end{rk}

\begin{ex}
If $r=13$ and $\mu=(3,2,0^2)$ then $d_{\mu}$ is given by $$\begin{array}{c} 
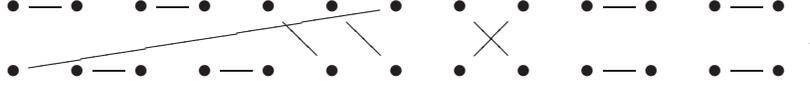
\begin{xy}
\xymatrix@!=0.01pc{ \bullet \tra[r] &  \bullet & \bullet \tra[r]  & \bullet  &  \bullet \tra[rd] &  \bullet \tra[rd]  & \bullet \tra[lllllld] & \bullet \tra[rd] & \bullet \tra[ld]  & \bullet \tra[r]  & \bullet & \bullet \tra[r]  & \bullet \\ \bullet  & \bullet \tra[r] &  \bullet & \bullet \tra[r] & \bullet  & \bullet  &  \bullet & \bullet & \bullet  & \bullet \tra[r]  & \bullet & \bullet \tra[r]  & \bullet   }
\end{xy} \end{array}.$$
\end{ex}

The following proposition studies eigenvalues connected to these canonical diagrams:

\begin{prop}
\label{prop eigenvals are roots of delta}
 Let $\rho: \brd \to M_n(F)$ be any $\brd$-representation, $d_{\mu} \in \brd$ be the canonical diagram corresponding to the cycle type $\mu=(q^{m_q},\ldots,2^{m_2},1^{m_1},0^{m_0})$ for some $q \in \N$ and let $l$ be the least common multiple of the natural numbers $\{ m_i \ | \ i=1,\ldots,q, m_i \neq 0 \}$.

Then the minimal polynomial of $\rho( d_{\mu})$ divides $x(x^{l}-\delta^{lm_0})$ and in particular the eigenvalues of $\rho( d_{\mu})$ are either zero or of the form $\delta^{m_0}u$ where $u$ is some $l$th root of unity.
\end{prop}

\begin{proof}
By Lemma~\ref{lem properties of ht} it is clear that the smallest
natural number $a$ such that $h_t(\gamma_k)^a=h_t(\gamma_k)$ is just
$l+1$. In particular, let $d_{\mu}$ be the canonical diagram of cycle type $\mu$. Hence, $d_\mu=h_t(\gamma_{\mu_1}) \otimes
\gamma_{\mu' \smallsetminus \mu_1} \otimes e^{\otimes m_0}$ where
$\mu'=(\mu_1,\mu_2,\ldots,\mu_s)$ is the partition formed by the
non-zero parts of $\mu$. Then the smallest power of $d_{\mu}$ such
that we get a scalar multiple of $d_\mu$ is precisely the order of
$\gamma_{\mu'}$, which is equal to $l$, plus $1$.

Now $d_{\mu}^{l+1}=\delta^{lm_0} d_{\mu}$. Because $\rho$ is an
algebra homomorphism we must have that
$\rho(d_{\mu})^{l+1}=\delta^{lm_0} \rho(d_{\mu})$, and thus the
minimal polynomial of $\rho(d_{\mu})$ divides
$x(x^{l}-\delta^{lm_0})$.  Hence the eigenvalues of $\rho(d_{\mu})$
are either zero or $l$th roots of  $\delta^{lm_0}$. But every such
root is precisely of the form $\delta^{m_0}u$ for some $l$th root of
unity $u$, as required.
\end{proof}

Recall that we have seen in Theorem~\ref{theo same ct implies same
eigenvalues} that if two diagrams have the same cycle type then they
have very similar minimal polynomials, the same non-zero
eigenvalues, and these eigenvalues have the same geometric
multiplicity. Thus we can deduce:

\begin{theo}
\label{thm same ct implies same char poly} Let $\rho: \brd \to
M_n(F)$ be any $\brd$-representation and suppose the diagrams
$d_1,d_2 \in \brd$ have the same cycle type. Then the matrices
$\rho(d_1)$ and $\rho(d_2)$ have the same characteristic polynomial
and in particular the same eigenvalues, including multiplicities.
\end{theo}

\begin{proof}
By Theorem~\ref{theo same ct implies same eigenvalues} we know that
the minimal polynomials of diagrams with the same cycle type only
differ by a power of the indeterminate. Thus, by
Proposition~\ref{prop eigenvals are roots of delta}, the minimal
polynomial of $\rho(d_1)$ and $\rho(d_2)$ both divide
$x^t(x^{l}-\delta^{lm_0})$ for some sufficiently large $t \in \N$.

Notice that the non-zero roots of this polynomial all have
multiplicity $1$ and are of the form $\delta^{m_0}u$ for some $l$th
root of unity $u$. Thus, by linear algebra, the size of any Jordan
block corresponding to a non-zero eigenvalue is equal to $1$. We may
conclude that the algebraic and geometric multiplicities of each
non-zero eigenvalue coincide, in other words, for each $0 \neq
\lambda \in F$ the dimension of the eigenspace corresponding to
$\lambda$ is the same as the multiplicity of $\lambda$ as a root of
the characteristic polynomial.

But by Theorem~\ref{theo same ct implies same eigenvalues} we know
that the geometric multiplicities for each non-zero eigenvalue are
the same among diagrams with the same cycle type. Thus, if two
diagrams have the same cycle type then the non-zero eigenvalues also
have the same algebraic multiplicities. But then the same must also
hold for the eigenvalues which are zero since the characteristic
polynomials of matrices of a fixed size have the same number of
roots.
\end{proof}

Let us record  a direct consequence for later use:
\begin{cor}
\label{cor sum of eigenvals of any diagram} Let $\rho: \brd \to
M_n(F)$ be any $\brd$-representation,  $d \in \brd$ be a diagram of
cycle type  $\mu=(\mu_1,\mu_2,\ldots,\mu_k,0^m)$ with $\mu_i \neq 0$
for $i=1,\ldots,k$. Suppose that $\mu'=(\mu_1,\mu_2,\ldots,\mu_k)$ is a
$p$-special partition. Then the eigenvalues of $\rho(d)$ are either
zero or of the form $\delta^{m_0}u$ where $u$ is some $l$th roots of
unity for some integer $l$ which divides the order of $\Sigma_r$ but is not divisible by $p$.
\end{cor}

\begin{proof}
 By Theorem~\ref{thm same ct implies same char poly},
we know that $\rho(d)$ has the same characteristic polynomial as $\rho(d_{\mu})$ where $d_{\mu}$ is the canonical diagram of cycle type $\mu$. So, in particular, $\rho(d)$ and $\rho(d_{\mu})$ have the same eigenvalues including multiplicities. But by Proposition~\ref{prop eigenvals
are roots of delta} it follows that each eigenvalue of $\rho(d_{\mu})$ is of the form
$\delta^{m_0} u$ for some $l$th root of unity $u$. Notice that, by construction, $l$ will precisely be the order of a permutation of cycle type $\mu'$. Thus, $l$ will divide the order of $\Sigma_{|\mu|}$ and hence the order of $\Sigma_r$. Also, since $\mu$ is $p$-special, we obtain that $l$ is not divisible by $p$. The result follows.
\end{proof}

\section{Implications for character theory}
We will now study what the results on eigenvalues of matrix
representations imply for character theory. In particular, we will use these to define modular characters for Brauer algebras which are analogues of Brauer character for groups. We will also obtain a theory of ordinary characters. In the semisimple,
complex case these were first studied by Arun Ram in
\cite{Ram}. He studied the characters of the generic Brauer algebra
and showed that apart from finitely many values of $\delta$ the
characters of $\brd$ can be obtained from the characters of the generic Brauer algebra by
specialisation, that is by evaluating $x$ at $\delta$. He also showed
that these characters are closely connected to the characters of the
symmetric group $\Sigma_r$. We are going to extend Ram's work to arbitrary
values of $\delta$, including the case when the Brauer algebra is
not semisimple. 
\subsection{Properties of ordinary characters}
We will start with defining what we mean by an ordinary character of the Brauer algebra:

\begin{defn}
Let $D$ be the standard basis of diagrams for $\brd$ and suppose
$\rho: \brd \to M_n(F)$ is a matrix representation of $\brd$. Then
the character associated to $\rho$, denoted $\chi_{\rho}$, is a
function $\chi_{\rho}: D \to F$ given by $$
\chi_{\rho}(d)=\tr(\rho(d)),$$ where $d \in D$ and $\tr$ denotes the
trace of a matrix.
\end{defn}

\begin{rk}
\begin{enumerate}[(i)]
\item Notice that we define characters for diagrams only. This follows the approach of group character theory where characters usually are defined on group elements instead of elements of the group algebra. Of course, one could extend this notion linearly to the entire algebra.
\item As usual the character does not depend on the choice of basis since the trace, and hence the character, is a class functions that is for any $a,b \in D$, we have $\tr( \rho(a)\rho(b)) = \tr (\rho(b)\rho(a))$.
\end{enumerate}
\end{rk}

Just as conjugate elements in a group have the same character we can immediately deduce:

\begin{theo}
\label{theo char const on gct} Let $\chi$ be a $\brd$-character and
suppose two diagrams $d_1,d_2 \in \brd$ have the same generalized
cycle type (defined in Definition~\ref{defnGCT}). Then the
corresponding character values are equal, that is $\chi(d_1) =
\chi(d_2)$.
\end{theo}

\begin{proof}
Let $\rho$ be a representation with associated character $\chi$. By
Theorem~\ref{theorem sigma conjugate iff gct the same}, we know that
$d_1$ and $d_2$ are $\Sigma_r$-conjugate, say by a permutation
$\sigma \in \Sigma_r$. Since the trace is a class function,
$\chi(d_2)= \tr(\rho(d_2)) = \tr\left( \rho(\sigma \inv) \rho(d_1) \rho(\sigma) \right) = \tr(\rho(d_1)) = \chi(d_1) $. 
\end{proof}

The next proposition summarizes the properties of ordinary
characters. They follow immediately from Theorem~\ref{thm same ct implies same char poly} and Theorem~\ref{thm char of arbitrary diagram}, respectively.

\begin{theo}
\label{theo properties of ordinary chars}
\begin{enumerate}[(i)]
\item Suppose $d_1, d_2 \in \brd$ are diagrams with the same cycle type $\mu$ and let $\chi$ be a $\brd$-character. Then $\chi(d_1)=\chi(d_2)=:\chi(\mu)$.
\item Let $(\alpha)$ be an arbitrary cycle type, with $\alpha$ a list of non-negative integers, and suppose $\chi$ is a $\brd$-character. Then we have in the notation of the first part that $$ \chi((\alpha,0) ) = \delta \chi((\alpha)).$$
\end{enumerate}

\end{theo}

\begin{rk}
Notice that in the special case when the Brauer algebra is semisimple these results on characters are implied by work of Ram, see Theorem 3.1 in \cite{Ram}.
\end{rk}

We will from now on adopt the following notation:
\begin{note}
For any $\brd$-character $\chi$ and any cycle type $\mu$ we denote by $\chi(\mu)$ the value $\chi$ takes on some, or equivalently all, diagrams of cycle type $\mu$.
\end{note}

\subsection{Definition of modular characters}
In this section we define modular characters for Brauer
algebras. The treatment will broadly follow the
definition of Brauer characters in Chapter 15 of \cite{Isaacs}.

Suppose we are given a matrix representation $\rho$ of a finite group $G$ over a field of prime characteristic $p$. Then one can define a function from the $p$-regular elements of $G$, that is elements with order coprime to $p$, to a field of characteristic $0$ whose properties resemble many of those of ordinary characters. This function is called the Brauer character of $\rho$ and can be defined for any $p$-regular element $g \in G$ in the following way: Determine the eigenvalues of $\rho(g)$, lift these to characteristic $0$ and take the sum of these lifted eigenvalues.  Since the eigenvalues of $\rho(g)$ will always be $k$th roots of unity of the field of characteristic $p$, the Brauer characters are well-defined.   This follows since $k$th roots of unity lift uniquely to $k$th roots of unity of the corresponding field of characteristic $0$, provided $k$ satisfies certain properties, see Proposition~\ref{prop roots of 1 lift uniquely}.

The approach for Brauer algebras has to be modified slightly because the eigenvalues of the corresponding matrices will not be roots of unity in general. In fact they might even be zero.
Furthermore, for ordinary characters we can always extend the definition of a character from the group to the group algebra by linear extension. Since the group algebra and the Brauer character values are over different fields, this extension will not work. Hence Brauer characters can only be evaluated at group elements.
This is the reason why we have made a similar adjustment for the ordinary characters of the Brauer algebra, by restricting the domain of the characters to the  basis of diagrams.

Notice that a symmetric group element is $p$-regular if and only if its cycle type is $p$-special. We will start by defining what the analogues of the $p$-regular  elements are:

\begin{defn}
A cycle type $\mu=(\ldots,2^{m_2},1^{m_1},0^{m_0})$ with $m_i \in \N_0$ is called $p$-special if the partition $\mu' = (\ldots,2^{m_2},1^{m_1})$ formed by the non-zero parts of $\mu$ is $p$-special. A diagram $d \in \brd$ is called $p$-special if its cycle type is $p$-special. Denote by $D_{p'}$ the subset of the diagram basis $D$ of $\brd$ consisting of $p$-special diagrams. As before, we also define $C$ to be the set of partitions of $r-2k$ for $0\leq k \leq [r/2]$ and $C_{p'}$ to be the subset of $C$ consisting of $p$-special partitions.
\end{defn}

 We are now ready to define modular characters for the Brauer algebra:

\begin{defn}
\label{defn brauer chars}
 Let $\rho: \brd \to M_n(F)$ be an $F$-representation of $\brd$ and let $D_{p'}$ be the set of all $p$-special diagrams in $\brd$. The modular character $\varphi_{\rho}$ associated to $\rho$ is a function $\varphi_{\rho}: D_{p'} \to S$ defined as follows: Let $d \in D_{p'}$ be any $p$-special diagram with cycle type $\mu=(\ldots,2^{m_2},1^{m_1},0^{m_0})$. Then $\rho(d)$ has eigenvalues $\delta^{m_0} u_1,\ldots,\delta^{m_0}u_n$ for some $l$th roots of unity $u_i \in F$. Here $l \in \N$ divides $|\Sigma_r|$ but is not divisible by $p$. Then $$ \varphi_{\rho}( d)  := \hat{\delta}^{m_0} (\hat{u}_1 +  \hat{u}_2 + \ldots + \hat{u}_n). $$
\end{defn}

\begin{rk}
Notice that there is a certain amount of ambiguity because of the choice of $\hat{\delta}$. However, taking into account that the semisimplicity of $\brd$ in characteristic $0$ depends heavily on the choice of $\delta$, the choice of $\hat{\delta}$ will give us some flexibility which can be useful in practice.
\end{rk}

\subsection{Properties of modular characters}
In this section we will state some properties which justify the definition of modular characters.  We can immediately deduce:

\begin{prop}
\label{prop modchar properties}

Let $\rho: \brd \to M_n(F)$ be an $F$-representation of $\brd$ and let $\varphi_{\rho}: D_{p'} \to S$ be the associated modular character. Then:
\begin{enumerate}[(i)]
\item The modular character  $\varphi_{\rho}$ is well-defined.
\item Equivalent representations afford the same modular character.
\item  If $d,d' \in D_{p'}$ have the same cycle type then $ \varphi_{\rho} (d) = \varphi_{\rho}(d')$. In particular this is true if they are $\Sigma_r$-conjugate.
\item Reducing the modular character $\varphi_{\rho}$ modulo $\Pi$ gives the ordinary trace character $\chi_{\rho}$:  $$ \overline{\varphi_{\rho} (d)} = \varphi_{\rho} (d) + \Pi = \chi_{\rho}(d) $$  for any diagram $d \in D_{p'}$.
\item Let $\pi \in \Sigma_r \subset D_{p'}$ be a $p$-special permutation diagram and denote by $\varphi^{\Sigma_r}_{\rho}$ the Brauer character corresponding to the representation $\rho$ restricted to the symmetric group $\Sigma_r$. Then $$\varphi^{\Sigma_r}_{\rho}(\pi) = \varphi_{\rho}(\pi).$$
\end{enumerate}
\end{prop}

\begin{proof}
\begin{enumerate}[(i)]
\item Once $\hat{\delta}$ is fixed, the only possible ambiguity is the choice of the lifted eigenvalues $\hat{u}_i$. But by Proposition~\ref{prop roots of 1 lift uniquely} these lifts are unique as well since  the natural number $l$ divides $|\Sigma_r|$ but is not divisible by $p$.
\item Suppose $\rho$ and $\rho'$ are equivalent representations and let $d$ be any diagram with cycle type $\mu=(\ldots,2^{m_2},1^{m_1},0^{m_0})$.  Then the matrices $\rho(d)$ and $\rho'(d)$ are conjugate and hence have the same eigenvalues. 

\item By Theorem~\ref{thm same ct implies same char poly}, the matrices $\rho(d)$ and $\rho(d')$ have the same non-zero eigenvalues, including multiplicities. Hence, by definition, they also the same modular character.  The second part then follows by Proposition~\ref{prop same gct implies same ct} and Theorem~\ref{theorem sigma conjugate iff gct the same}.
\item This part follows straight from the definition because the modular character was precisely defined as the sum of the lifted eigenvalues. Hence reducing this lift modulo $\Pi$ will recover the original sum of eigenvalues which is equal to the trace.
\item Since $\pi$ is a permutation, its cycle type $\mu=(\ldots,2^{m_2},1^{m_1},0^{m_0})$  satisfies $m_0=0$. Thus, if $u_1,\ldots,u_n$ are the eigenvalues of $\rho(\pi)$ then  $$\varphi_{\rho}(\pi) = \left( \hat{u}_1+ \ldots + \hat{u}_n \right) = \varphi^{\Sigma_r}_{\rho}(\pi).$$ The last equality follows by the definition of the Brauer character.
\end{enumerate}

\end{proof}

\begin{rk}
\label{rk mod char lab by ct}
\begin{enumerate}[(i)]
\item Since modular characters are constant on diagrams with the same cycle type, we will frequently view modular characters as functions of $p$-special cycle types, as opposed to $p$-special diagrams. Therefore, if $d \in \brd$ is a diagram of cycle type $\mu$, we write $\varphi_{\rho}(\mu):=\varphi_{\rho}(d)$.
\item Straight from the definition, we get the following property which is similar to Theorem~\ref{theo properties of ordinary chars}: If $(\alpha)$ is an arbitrary cycle type, with $\alpha$ a list of non-negative integers, then we have $$ \varphi_{\rho}((\alpha,0) ) = \hat{\delta} \varphi_{\rho}((\alpha)).$$
\end{enumerate}
\end{rk}

The following properties are well known for Brauer characters of finite groups and their proof for Brauer algebras is virtually identical:

\begin{prop}[See for example Chapter 40 in \cite{Kuelshammer}]
\label{prop same cf iff same modular char}
\begin{enumerate}[(i)]
\item Let $\{ S_i \}_{i \in I}$ be a complete set of irreducible $\brd$-modules and denote by $\{ \varphi_{i }\}_{i \in I}$ the corresponding irreducible modular characters. Then $\{ \varphi_{i }\}_{i \in I}$ is a linearly independent set.
\item  If $N \leq M$ are $\brd$-modules with associated modular characters $\varphi_M$ and $\varphi_N$, then $$ \varphi_{M} = \varphi_{N} +\varphi_{M/N}.$$ In particular, two modules have the same modular character if and only if they have the same composition factors.
\end{enumerate}
\end{prop}

\section{Computation of characters}
\label{section computation of characters}
We will now turn our attention to irreducible characters, both ordinary and modular, and show how we can relate them to irreducible characters of smaller Brauer algebras and symmetric groups. In particular, we will show that the character table of the symmetric group $\Sigma_r$ and the character table of the smaller Brauer algebra $B_{r-2} (\delta)$ are submatrices of the character table of $\brd$. This phenomenon was already shown in Corollaries 5.2-5.4 of \cite{Ram} in the case of the ordinary character table of semisimple Brauer algebras. His proof uses the irreducible characters of the orthogonal and symplectic groups and Schur-Weyl duality. We are going to generalize his results to the non-semisimple case and also show that they hold for the modular character table by comparing eigenvalues of matrix representations of certain cell modules.

\subsection{Definition of ordinary and modular character tables}
We will start by defining character tables both in the ordinary as well as in the modular case. We will first define the sets labelling the character tables:

\begin{note}
Recall that $C$ is the set of all partitions of $r-2k$ for $0 \leq k \leq [r/2]$, and $C_{p'}$ and $C_p$ are the subsets of $C$ consisting of all $p$-special and $p$-regular partitions, respectively, where $p$ is the characteristic of the field $F$. For any $\lambda \in C_p$ and $t=\frac{r-|\lambda|}{2}$, denote by  $\chi^{(t,\lambda)}: C \to F$ the ordinary character corresponding to the simple module $D^{(t,\lambda)}$ and by $\varphi^{(t,\lambda)}: C_{p'} \to S$ the modular character corresponding to the simple module $D^{(t,\lambda)}$. Notice that we view characters as a function of the cycle type, see Remark~\ref{rk mod char lab by ct}.
\end{note}

Our definition of character tables will be different from Ram's (see the end of Section $5$ in \cite{Ram}) and what one might expect from group character theory. Firstly, not all character values will be included in the character tables. However, it will be possible to compute all character values from the included values using Theorem~\ref{theo properties of ordinary chars} and Remark~\ref{rk mod char lab by ct}, respectively. For example, instead of recording $\chi^{(t,\lambda)}((2))$ and  $\chi^{(t,\lambda)}((2,0)) = \delta \cdot \chi^{(t,\lambda)}((2))$, we just record $\chi^{(t,\lambda)}((2))$.  Consequently, the character tables are independent of $\delta$. Secondly, ordinary group character tables are normally only defined in characteristic $0$. Because we particularly want to emphasize the relation of the characters to smaller Brauer algebras and symmetric groups, we will also define character tables for the case when the Brauer algebra is not semisimple and even in positive characteristic. Therefore, the rows will be labelled by the actual simple modules $D^{(t,\lambda)}$, instead of the generic simple modules $S^{(t,\lambda)}$.

\begin{defn}
\begin{enumerate}[(i)]
\item The ordinary character table of $\brd$, denoted by $\Xi_{\brd}$, is a matrix with rows indexed by $C_{p}$ and columns indexed by $C$. Furthermore, the $(\lambda,\mu)$th entry is given by $ \chi^{(t,\lambda)}(\mu)$ for $\lambda \in C_{p}, \mu \in C$ and $t=\frac{r-|\lambda|}{2}$.
\item The modular character table of $\brd$, denoted by $\Phi_{\brd}$, is a matrix with rows indexed by the set of $p$-regular partitions $C_{p}$ and columns indexed by the set of $p$-special partitions $C_{p'}$. Furthermore, the $(\lambda,\mu)$th entry is given by $\varphi^{(t,\lambda)}(\mu)$ for $\lambda \in C_{p}, \mu \in C_{p'}$ and $t=\frac{r-|\lambda|}{2}$.
\end{enumerate}
\end{defn}

\begin{rk}
The ordinary character table, together with Theorem~\ref{thm char of arbitrary diagram}, does indeed encode all the information needed to calculate the irreducible characters. This follows because
\begin{enumerate}[(i)]
\item  the set $C_{p}$ is the indexing set for the simple $\brd$-modules, see Remark~\ref{rk indexing sets of simple modules}. 
\item  the set $C$ is the indexing set for the cycle types with no $0$s, see Remark~\ref{rk which ct can occur}.
\end{enumerate}
A similar statement holds for modular characters.
\end{rk}

We get the following standard properties of the modular character table:

\begin{lem}
\begin{enumerate}[(i)]
\item The determinant of $\Phi$ is not in $\Pi$ and hence not a multiple of the characteristic $p$.
\item If $\Xi_{p'}$ denotes the ordinary character table $\Xi$ with columns restricted to the $p$-special partitions then $$\Xi_{p'} = D\Phi$$ where $D$ denotes the decomposition matrix of $\brd$.
\end{enumerate}
\end{lem}

\begin{proof}
\begin{enumerate}[(i)]
\item Reducing the Brauer characters modulo $\Pi$ gives precisely the ordinary characters which are linearly independent. Since the determinant of the  reduced matrix is just the reduction of the determinant  of $\Phi$ modulo $\Pi$, the result follows.
\item This part is clear from the definition of the decomposition matrix and the fact that two modules have the same character if and only if they have the same composition factors, see Proposition~\ref{prop same cf iff same modular char}.
\end{enumerate}
\end{proof}

\subsection{Structure of the character tables}
In this section, we will show that the character table of the symmetric group $\Sigma_r$ and the character table of the smaller Brauer algebra $B_{r-2} (\delta)$ are submatrices of the character table of $\brd$. We will first translate this into a statement on eigenvalues.

\begin{prop}
\label{prop same ct diff br same mod char}
Let $d_t \in \brd$ be a  diagram with cycle type  $\mu$ where  $|\mu|  \leq r-2$ and let $1 \leq t \in \N$. Let $d_{t-1} \in B_{r-2}(\delta)$ be a diagram with the same cycle type $\mu$. Suppose also that for $\lambda$ a partition of $r-2t$ the action of $d_t$ on the $\brd$-module $S^{(t,\lambda)}$ is given by some matrix $R_t$ and the action of $d_{t-1}$ on the $B_{r-2}(\delta)$-module $S^{(t-1,\lambda)}$ is given by some matrix $R_{t-1}$. Then $R_t$ and $R_{t-1}$ have the same non-zero eigenvalues including multiplicities.
\end{prop}

\begin{proof}
Recall the construction of the modules $S^{(t,\lambda)}$ from Section~\ref{section defn of cell modules}. We will show that there are bases of $S^{(t,\lambda)}=S^{\lambda} \otimes V_{t,r}$ and $S^{(t-1,\lambda)} = S^{\lambda} \otimes V_{t-1,r-2}$ such that, with respect to these bases, $$ R_t = \left( \begin{array}{cc}
R_{t-1} & A  \\
 0 & 0
\end{array} \right) $$ for some matrix $A$.

 Let $\{e_1,e_2,\ldots,e_q \}$ be a basis of the Specht module $S^{\lambda}$ and let $\mathcal{B}_{t-1,r-2}=\{ v_1,\ldots,v_w\}$ be the standard basis of partial diagrams for the vector space $V_{t-1,r-2}$. For each partial diagram $v_i \in V_{t-1,r-2}$ we can define an element $\hat{v_i} \in V_{t,r}$ in the following way: Add two dots to the left of $v_i$ and connect these two dots by an arc, that is $$\hat{v_i} = \begin{array}{c} \begin{xy}  \xymatrix@!=0.01pc{ \bullet \tra[r] & \bullet  & \tra[rr] & & v_i \tra[rr] & & }  \end{xy} \end{array}. $$ Any partial diagram in $V_{r,t}$ in which dots $1$ and $2$ are connected by an arc arises in this way. Extend the set $\hat{\mathcal{B}}_{t-1,r-2}=\{ \hat{v}_1,\ldots,\hat{v}_w \}$ to a basis $\mathcal{B}_{t,r}$ of the vector space $V_{t,r}$.

Let $d_{\mu}$ be the canonical diagram of cycle type $\mu$ as defined in Definition~\ref{defn canonical diagram of cycle type}. Denote by $d_{\mu} \sta$ the diagram obtained by horizontally flipping $d_{\mu}$. Notice that this has cycle type $\mu$ since the direction of the arrows in each connected component are reversed.  By Theorem~\ref{thm same ct implies same char poly}, we may assume that $d_t$ and $d_{t-1}$ are any diagrams with cycle type $\mu$. Therefore, for convenience, we will choose $d_t=d_{\mu}^*$. Thus $d_t \in \brd$ is of the form $$d_t = \begin{array}{c} \begin{xy} \xymatrix@!=0.01pc{  \bullet  &\bullet \tra[r] &  \bullet & & & & \\ & &  & \cdots & \bullet^x \tra[dll] & \cdots &\bullet^y \tra[ullllll] \\ \bullet \tra[r] & \bullet    & \bullet   & & & & } \end{xy} \end{array}. $$ Also, we will choose $d_{t-1}$ to be the following diagram on $2(r-2)$ dots:
$$d_{t-1} = \begin{array}{c} \begin{xy} \xymatrix@!=0.01pc{    \bullet & & & & \\  & \cdots & \bullet^x \tra[dll] & \cdots &\bullet^y \tra[ullll] \\  \bullet   & & & & } \end{xy} \end{array}.$$ This is obtained from $d_t$ by deleting the first two dots in each row and connecting the third dot of the top row of $d_t$ to $y$. Notice that $d_t$ and $d_{t-1}$ have the same cycle type $\mu$.

 Thus, by the definition of multiplication of the partial diagram $v \in V_{t,r}$ by the diagram $d_t$, we have that $v \circ d_t$ will always have an arc between dots $1$ and $2$. Hence, $v \circ d_t \in \hat{\mathcal{B}}_{t-1,r-2}$ for any $v \in V_{t,r}$. We can conclude that the matrix $R_t$ is of the form $$R_t = \left( \begin{array}{cc}
\ast & \ast \\
0 & 0
\end{array} \right)$$ with respect to the basis $\{ e_i \otimes b \ | \ i=1,\ldots,q, \ b \in \mathcal{B}_{t,r} \}$.

It remains to show that the top left part of $R_t$ is the same as $R_{t-1}$. This is the same as showing that the action of $d_t$ on $\langle \hat{\mathcal{B}}_{t-1,r-2} \rangle_F$, including the induced permutation of this action, are the same as the action of $d_{t-1}$ on $V_{t-1,r-2}$.

Now for any $i=1,\ldots,w$, the product $\hat{v}_i d_t$ is given by:
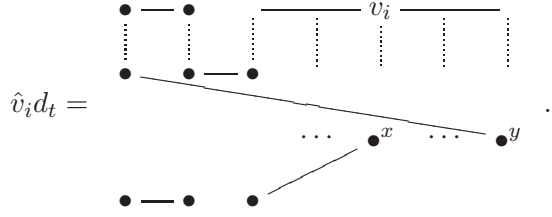
\begin{align*} 
\hat{v}_i d_t=  \begin{array}{c} \begin{xy} \xymatrix@!=0.01pc{ \bullet \tra[r] \dta[d] & \bullet \dta[d]  & \tra[rr]  \dta[d]  &  \dta[d]  & v_i \tra[rr]  \dta[d]  &  \dta[d]  &  \dta[d]  \\ \bullet  &\bullet \tra[r] &  \bullet & & & & \\ & &  & \cdots & \bullet^x \tra[dll] & \cdots &\bullet^y \tra[ullllll] \\ \bullet \tra[r] & \bullet    & \bullet   & & & & } \end{xy} \end{array}. \end{align*}
It follows that
 $$\hat{v}_i  d_t = \begin{array}{c} \begin{xy}  \xymatrix@!=0.01pc{ \bullet \tra[r] & \bullet  & \tra[rr] & & v_i d_{t-1} \tra[rr] & & }  \end{xy} \end{array}, $$ that is, $\hat{v}_i d_t = \widehat{v_i d_{t-1}}$.
Therefore the result follows.
\end{proof}

Consider the ordinary and modular character table of the $\brd$-modules $S^{(t,\lambda)}$, defined in the obvious way. Then Proposition~\ref{prop same ct diff br same mod char} implies that the corresponding character table of the smaller Brauer algebra $B_{r-2}(\delta)$ sits at the bottom right. A similar statement is known for the decomposition matrix, see Proposition 5.2 in \cite{HartmannPaget}. Thus, we can deduce that this special form also holds for the character table for the simple modules $D^{(t,\lambda)}$.

\begin{theo}
\label{thm modular char of  quasipartitions} \label{thm br char of quasipartitions}
\begin{enumerate}[(i)]
\item Let $\lambda$ be a partition and let $t$ and $t'$ be non-negative integers with $t \geq t'$. Suppose  $\mu$ is a cycle type for both $B_{|\lambda|+2t'}(\delta)$ and $B_{|\lambda|+2t}(\delta)$. Then $$\chi^{(t,\lambda)}(\mu)=\chi^{(t',\lambda)}(\mu).$$
Equivalently, there is an ordering of the rows and columns of the character table $ \Xi_{\brd}$  such that $$ \Xi_{\brd} =\left( \begin{matrix}
  \Xi_{\Sigma_r}& 0 \\
 A &  \Xi_{B_{r-2}(\delta)}
\end{matrix} \right).$$ Here  $\Xi_{\Sigma_r}$ denotes the character table of the symmetric group $\Sigma_r$ and $A$ is some matrix.
\item  Let $\lambda$ be a $p$-special partition and let $t$ and $t'$ be non-negative integers with $t \geq t'$. Suppose  $\mu$ is a $p$-special cycle type for both $B_{|\lambda|+2t'}(\delta)$ and $B_{|\lambda|+2t}(\delta)$. Then $$\varphi^{(t,\lambda)}(\mu)=\varphi^{(t',\lambda)}(\mu).$$
Equivalently, there is an ordering of the rows and columns of the character table $ \Phi_{\brd}$  such that $$ \Phi_{\brd} =\left( \begin{matrix}
  \Phi_{\Sigma_r}& 0 \\
 A &  \Phi_{B_{r-2}(\delta)}
\end{matrix} \right).$$ Here  $\Phi_{\Sigma_r}$ denotes the Brauer character table of the symmetric group $\Sigma_r$ and $A$ is some matrix.
\end{enumerate}
\end{theo}

\begin{proof}
We will start with the case of the ordinary character table $\Xi_{\brd}$. By definition, the rows are labelled by the set $C_{p}$ and the columns are labelled by the set $C$.
Order $C$, and thus $C_{p} \subseteq C$,  in the following way: If $\lambda$ is a partition of $r-2t$ and $\mu$ is a partition of $r-2l$ then $\lambda \geq \mu $ whenever $t<l$. If $l=t$ order the partitions lexicographically. In particular, for fixed $t$ the partitions of $r-2t$ are ordered in the same way as they are for the character table of the symmetric group, denoted $\Xi_{\Sigma_{r-2t}}$.

Recall the definition of $D_{\brd}$ from Section~\ref{section defining decomposition matrices}.
Define the partial decomposition matrix $D'_{\brd}$ to be the square matrix obtained from $D_{\brd}$ by just taking all rows indexed by $C_{p}$. By Theorem~\ref{theo form of dec matrix} we know that $D'_{\brd}$  is of the form $$D'_{\brd} = \left( \begin{array}{cc}
 D'_{F\Sigma_r} & 0  \\
 A' &  D'_{B_{r-2}(\delta)}
\end{array} \right).$$  Here $D'_{F\Sigma_r}$ and $D'_{B_{r-2}(\delta)}$ are the partial decomposition matrices of $F\Sigma_r$ and $B_{r-2}(\delta)$, respectively, and $A'$ is some matrix. Notice that $D'_{\brd}$ is lower-triangular and each diagonal entry is $1$, so that $D'_{\brd}$ is invertible.

Let $X_{\brd}$ be the matrix with rows indexed by $C_{p}$ and columns indexed by $C$ defined as follows:  For $\lambda \in C_{p}$ a $p$-regular partition of $r-2t$ and $\mu \in C$, the $(\lambda,\mu)$th entry is given by  $ \chi^{(t,\lambda)}(\mu)$. Thus $X_{\brd}$ is the character table of the $\brd$-cell modules $S^{(t,\lambda)}$ corresponding to $p$-regular partitions.

It is clear that all diagrams on $r$ dots of cycle type a partition of $r$ must be permutations. But since $S^{(0,\lambda)} \res =S^{\lambda}$, the action of such a permutation on $S^{(0,\lambda)}$ is the same as the action on  $S^{\lambda}$. So in the case when both $\lambda$ and $\mu$ are partitions of $r$, the corresponding character value is just the ordinary symmetric group character value. Therefore, the top left corner of $X_{\brd}$  is precisely the ordinary character table of the symmetric group but restricted to $p$-regular partitions. Using this and Proposition~\ref{prop same ct diff br same mod char}, we can deduce that $$ X_{\brd}= \left( \begin{array}{cc}
 X_{\Sigma_r}&0  \\
 Q & X_{B_{r-2}(\delta)}
\end{array} \right).$$ Here $X_{\Sigma_r}$ and $X_{B_{r-2}(\delta)}$ are defined analogously to $X_{\brd}$ and $Q$ is some matrix.  Notice that for every $p$-regular partition $\lambda$, the character of the cell module  $S^{(t,\lambda)}$ is given by $$\sum_{\theta \in C_{p}} [ S^{(t,\lambda)}:  D^{(k,\theta)}] \chi^{(k,\theta)}$$ where $2t+|\lambda| = 2k+|\theta|$.  It follows that $$ X_{\brd} = D'_{\brd} \Xi_{\brd}.$$ Substituting, we have $$  \left( \begin{array}{cc}
 X_{\Sigma_r}&0  \\
 Q & X_{B_{r-2}(\delta)}
\end{array} \right) = \left( \begin{array}{cc}
 D'_{F\Sigma_r} & 0  \\
 A' &  D'_{B_{r-2}(\delta)}
\end{array} \right)   \Xi_{\brd}.$$ Since $D_{\brd}'$ is invertible by above, we obtain $$ \Xi_{\brd}=  \left( \begin{array}{cc}
 (D'_{F\Sigma_r}) \inv X_{F\Sigma_r}  & 0  \\
 A &    (D'_{B_{r-2}(\delta)})\inv X_{B_{r-2}(\delta)}
\end{array} \right)  =  \left( \begin{array}{cc}
 \Xi_{\Sigma_r} & 0  \\
 A &  \Xi_{B_{r-2}(\delta)}
\end{array} \right) $$ where $A$ is some matrix, as required.

In the second case the proof is exactly the same as the first part, except that one needs to substitute $\Phi$ for $\Xi$ and $\varphi$ for $\chi$.
\end{proof}

\begin{rk}
Notice that in the special case when $\brd$ is semisimple, part $(i)$ has been shown by Ram, see Corollaries 5.2-5.4 of \cite{Ram}. 
\end{rk}

The theorem justifies the following notation:

\begin{note}
For any partition  $\lambda$ of $r-2t$  and any cycle type $\mu$ set $\chi^{\lambda}(\mu):=\chi^{(t,\lambda)}(\mu)$ for any or equivalently all choices of $t$.
\end{note}

The theorem implies that in order to compute the ordinary and modular character table of the Brauer algebra one only needs two ingredients: the character table of the symmetric groups $\Sigma_r, \Sigma_{r-2},\ldots,\Sigma_1 \backslash \Sigma_0$ and the composition factors of the restriction of the simple $\brd$-modules $D^{(t,\lambda)}$ to the symmetric group $\Sigma_r$. This follows by observing that, with the chosen ordering, the first columns of the character tables correspond to partitions of $r$. By definition of the cycle type, diagrams corresponding to such a partition have no horizontal arcs. Furthermore, every diagram with no horizontal arcs correspond to a partition of $r$. Therefore these first columns of the character tables are simply labelled by conjugacy classes of permutations. So, in order to compute the matrix $A$, one only needs to compute the character values of permutation diagrams. It is thus sufficient to consider the restriction of the corresponding simple $\brd$-module to the symmetric group algebra $F\Sigma_r$. When the Brauer algebra is semisimple this restriction is known, see Theorem 4.1 in \cite{HW3}. 

We will now state another corollary of Theorem~\ref{thm modular char of  quasipartitions} which generalizes a well known result for symmetric groups, see for example Corollary 22.26 in \cite{JamesLiebeck}. Let $p$ be prime. If $g \in \Sigma_r$ then there exist unique $x,y \in \Sigma_r$ with the following properties: $g=xy=yx$, $x$ has order a power of $p$ and $y$ has order coprime to $p$. The element $y$ is called the $p'$-part of $g$. The result says that for an ordinary $\Sigma_r$-character $\chi$, we have $\chi(g) \equiv \chi(y) \mod p$. Here we use the fact that character values of symmetric groups are always integers. 

The $p'$-part of a permutation in cycle notation is obtained by ignoring all cycles of length a multiple of $p$. Let us examine what this means for the cycle type. If $g$ has cycle type $\mu=(\mu_1,\ldots,\mu_k)$ for some $k \in \N$, then the cycle type of $y$ is obtained by replacing each $\mu_i$ that is divisible by $p$ by $1^{\mu_i}$. The partition defined by this procedure is called the $p'$-part of $\mu$ and is denoted by $\mu_{p'}$. We extend this naturally to partitions including $0$s.

\begin{prop}
\label{prop character only depends on p' part}
Let $\chi$ be an ordinary  $\brd$-character and suppose $\mu$ is a cycle type for $\brd$. Then $\chi(\mu)=\chi(\mu_{p'})$.
\end{prop}

\begin{proof}
It is sufficient to prove this for an arbitrary irreducible ordinary character $\chi^{\lambda}$ for $\lambda \in C$. If $|\mu|<|\lambda|$ then  $\chi^{\lambda} (\mu) = 0$ by Theorem~\ref{thm modular char of  quasipartitions} and we are done. Let $\mu'$ be the partition obtained from $\mu$ by deleting all zeros. Suppose also that $\mu$ has $k$ parts equal to $0$ for some $k \in \N_0$. Thus  $\chi^{\lambda} (\mu) = \delta^k \chi^{\lambda} (\mu')$ by Theorem~\ref{theo properties of ordinary chars}. If $|\mu| \geq |\lambda|$ then we can deduce from Theorem~\ref{thm modular char of  quasipartitions}  and by considering the Brauer algebra $B_{|\mu|}(\delta)$ that $\chi^{\lambda} (\mu) = \delta^k \chi^{\lambda}_{\Sigma_{|\mu|}} (\mu')$. Here $\chi^{\lambda}_{\Sigma_{|\mu|}}$ denotes the $\Sigma_{|\mu|}$-character obtained by restriction of $\chi^{\lambda}$ to  $\Sigma_{|\mu|}$. Thus, the claim follows by the corresponding result for symmetric groups.
\end{proof}

\begin{rk}
\label{rk linear indep of rows of char table}
The proposition implies that the columns of the character table $\Xi$ are linearly dependent in characteristic $p$. However, if we restrict the indexing set of the columns of $\Xi$ to $p$-special partitions, then the resulting table will be square as there is a $1-1$ correspondence between $p$-regular partitions and $p$-special partitions. Thus the columns will be linearly independent since certainly the rows are.
\end{rk}

\section{Conjugacy via characters of representations}
\label{section chi conjugacy}
We are now finally in a position to discuss the third kind of conjugation defined in \cite{Mazorchuk}, namely $\chi$-conjugacy. We will determine the $\chi$-conjugacy classes using the character theory developed in the previous sections. Since characters are class functions, it is clear that if two elements of $H$ are $(uv,vu)$-conjugate then they are also $\chi$-conjugate. We will exploit this to give a full description of the $(uv,vu)$-conjugacy classes. 

In contrast to the $(uv,vu)$-conjugacy classes, the $\chi$-conjugacy classes do actually depend on the characteristic of the ground field:

\begin{theo}
\label{theo determination of chi conjugacy}
Let $F$ be a field of characteristic $p \geq0$. Suppose $d_1,d_2 \in H$ have cycle type $(\alpha,0^l)$ and $(\beta,0^m)$, respectively, for some non-increasing lists of positive integers $\alpha,\beta$ and some $l,m \in \N_0$.  Let  $(\alpha_{p'})$ and $(\beta_{p'})$ be the $p'$-parts of the partitions $(\alpha)$ and $(\beta)$, respectively. Then:
\begin{enumerate}[(i)]
\item  If $\delta$ is non-zero and not a root of unity then $d_1$ and $d_2$ are $\chi$-conjugate if and only if  $(\alpha_{p'},0^l) = (\beta_{p'},0^m)$.
\item If $\delta$  has order $n$ for some $n \in \N$ then $d_1$ and $d_2$ are $\chi$-conjugate if and only if $(\alpha_{p'}) = (\beta_{p'})$ and $l \equiv m \mod n$.
\item If $\delta=0$  then $d_1$ and $d_2$ are $\chi$-conjugate

 if and only if either $l,m \neq 0$ or both $l=m=0$ and $(\alpha_{p'}) = (\beta_{p'})$.
\end{enumerate}

\end{theo}

\begin{proof}
($\Longleftarrow$) Let $\chi$ be an ordinary $\brd$-character. Then we deduce from Theorem~\ref{theo properties of ordinary chars} and Proposition~\ref{prop character only depends on p' part} that \begin{align*} \chi( (\alpha,0^l))  =\delta^l\chi( (\alpha))=\delta^l\chi( (\alpha_{p'}))  \text{ \ and \ }
 \chi( (\beta,0^m))  =\delta^m\chi( (\beta))=\delta^m\chi( (\beta_{p'})). \end{align*}
Therefore, it follows in each of the three case that  $\chi(d_1)= \chi( (\alpha,0^l))= \chi( (\beta,0^m))=\chi(d_2)$ so that $d_1$ and $d_2$ are $\chi$-conjugate, as required.

($\Longrightarrow$)
Suppose without loss of generality that  $|(\alpha)| \geq |(\beta)|$. Let $t=\frac{r-|(\alpha)|}{2}$ and  $k=\frac{r-|(\beta)|}{2}$. For each partition $\lambda$ of $r-2t$, we have by Theorem~\ref{theo properties of ordinary chars} and Proposition~\ref{prop character only depends on p' part} and by Theorem~\ref{thm modular char of  quasipartitions}  that \begin{align*} \chi^{\lambda}( (\alpha,0^l)) & =\delta^l\chi^{\lambda}( (\alpha))=\delta^l\chi_{\Sigma_{r-2t}}^{\lambda}( (\alpha_{p'})), \\
 \chi^{\lambda}( (\beta,0^m)) & =\delta^m\chi^{\lambda}( (\beta))=\delta^m\chi^{\lambda}( (\beta_{p'})). \end{align*}
Here $\chi_{\Sigma_{r-2t}}^{\lambda}$ denotes the character corresponding to the $\Sigma_{r-2t}$-module $S^{\lambda}$.
Notice that whenever  $|(\alpha)| > |(\beta)|$ then Theorem~\ref{thm modular char of  quasipartitions} implies $ \chi^{\lambda}( (\beta))=0$. With this in mind, let us study each of the above cases in turn:

\begin{enumerate}[(i)]
\item If $|(\alpha)| > |(\beta)|$, we deduce that $\delta^l \chi_{\Sigma_{r-2t}}^{\lambda}( (\alpha_{p'}))=0$ and hence $ \chi_{\Sigma_{r-2t}}^{\lambda}( (\alpha_{p'}))=0$ for all partitions $\lambda$ of $r-2t$. But this contradicts the linear independence of the columns of the character table of $\Sigma_{r-2t}$ corresponding to $p$-special partitions, see Remark~\ref{rk linear indep of rows of char table}.

If $|(\alpha)| = |(\beta)|$ then we know by assumption that $$\delta^m\chi^{\lambda}( (\beta_{p'})) = \delta^m\chi_{\Sigma_{r-2t}}^{\lambda}( (\beta_{p'})) =  \delta^l \chi_{\Sigma_{r-2t}}^{\lambda}( (\alpha_{p'})).$$ If  $(\alpha_{p'}) \neq (\beta_{p'})$ this again contradicts the linear independence of the columns of the character table of $\Sigma_{r-2t}$ corresponding to $p$-special partitions as $\delta \neq 0$.  
Consider the case $(\alpha_{p'}) = (\beta_{p'})$ but $l \neq m$. Notice that there will be some partition $\lambda$ of $r-2t$ such that  $\chi_{\Sigma_{r-2t}}^{\lambda}( (\alpha_{p'})) \neq 0$ by a similar linear independence argument as above. We thus deduce that $\delta^l=\delta^m$ which contradicts the fact that $\delta$ is not a root of unity.

\item  In this case we can essentially use the same argument as in the first part except that we might have $\delta^l=\delta^m$ but $l\neq m$. However it is clear that this holds if and only if $l \equiv m \mod n$.

\item We have to distinguish two cases:

\begin{enumerate}
\item \textit{Assumption: $l=0$ and either $m \neq 0$ or $(\alpha_{p'}) \neq (\beta_{p'})$.}  We can conclude $\chi_{\Sigma_{r-2t}}^{\lambda} ((\alpha_{p'})) =  \delta^m \chi^{\lambda} ((\beta_{p'}))$. Therefore, if either $m \neq 0$ or $(\alpha_{p'}) \neq (\beta_{p'})$ we get a contradiction using that $\delta =0$ and similar arguments as above.

\item \textit{Assumption: $m=0$  and $l \neq 0$.}  If $\mu$ is any partition of $r-2k$ then we obtain $0=\delta^l \chi^{\mu} ((\alpha_{p'})) =  \chi^{\mu} ((\beta_{p'}))$. Thus we get a contradiction to  the linear independence of the columns of the character table of $\Sigma_{r-2k}$ corresponding to $p$-special partitions.

\end{enumerate}

\end{enumerate}
\end{proof}

As a consequence, we can determine the $(uv,vu)$-conjugacy classes:

\begin{cor}
\label{cor determination of uvvu conjugacy}
Let  $d_1,d_2 \in H$ with cycle type $(\alpha,0^l)$ and $(\beta,0^m)$, respectively, for some non-increasing lists of positive integers $\alpha,\beta$ and some $l,m \in \N_0$.  Then:
\begin{enumerate}[(i)]
\item  If $\delta$ is non-zero and not a root of unity then $d_1$ and $d_2$ are $(uv,vu)$-conjugate if and only if  they have the same cycle type.
\item If $\delta$  has order $n$ for some $n \in \N$ then $d_1$ and $d_2$ are $(uv,vu)$-conjugate if and only if $(\alpha) = (\beta)$ and $l \equiv m \mod n$.
\item If $\delta=0$  then $d_1$ and $d_2$ are $(uv,vu)$-conjugate

 if and only if either $l,m \neq 0$ or both $l=m=0$ and $(\alpha) = (\beta)$.
\end{enumerate}
\end{cor}

\begin{proof}
($\Longrightarrow$) If $d_1$ and $d_2$ are $(uv,vu)$-conjugate then they are in particular $\chi$-conjugate over a field of characteristic $0$. Thus one direction follows by Theorem~\ref{theo determination of chi conjugacy}.\\
($\Longleftarrow$) This direction follows immediately by Theorems~\ref{theo ct determined uv conjugation} and~\ref{thm char of arbitrary diagram}.
\end{proof}

\end{document}